\documentclass[12pt]{article}
\usepackage{amsmath, amssymb, amsthm, amsfonts}
\usepackage{tikz}
\usepackage[matrix,arrow,curve,frame]{xy}
\usepackage[colorlinks=true,citecolor=red,linkcolor=blue]{hyperref}

\input xy
\xyoption{all}

\swapnumbers
\numberwithin{equation}{section}

\theoremstyle{plain}
\newtheorem{theorem}[subsubsection]{Theorem}
\newtheorem{lemma}[subsubsection]{Lemma}
\newtheorem{prop}[subsubsection]{Proposition}
\newtheorem{cor}[subsubsection]{Corollary}

\theoremstyle{definition}
\newtheorem{defn}[subsubsection]{Definition}

\newtheorem{remark}[subsubsection]{Remark}

\def\CC{\mathbb{C}}

\def\NN{\mathbb{N}}

\def\RR{\mathbb{R}}

\def\ZZ{\mathbb{Z}}

\newcommand\cA{\mathcal{A}}
\newcommand\cB{\mathcal{B}}

\newcommand\cE{\mathcal{E}}
\newcommand\cF{\mathcal{F}}
\newcommand\cG{\mathcal{G}}

\newcommand\cM{\mathcal{M}}

\newcommand\cO{\mathcal{O}}

\newcommand\cU{\mathcal{U}}

\newcommand\cW{\mathcal{W}}

\newcommand\cY{\mathcal{Y}}

\newcommand\frW{\mathfrak{W}}

\newcommand\frg{\mathfrak{g}}

\newcommand{\coker}{\textup{coker}}

\newcommand\id{\textup{id}}
\renewcommand{\Im}{\textup{Im}}
\newcommand{\Ind}{\textup{Ind}}

\newcommand\Rep{\textup{Rep}}

\newcommand\Span{\textup{Span}}

\newcommand\Hom{\textup{Hom}}

\newcommand{\Ext}{\textup{Ext}}

\newcommand\gl{\mathfrak{gl}}

\renewcommand\sl{\mathfrak{sl}}

\newcommand{\btimes}{\boxtimes}
\newcommand{\jiao}[1]{\langle{#1}\rangle}

\newcommand\quash[1]{}
\newcommand\xr{\xrightarrow}

\newcommand\one{\mathbf{1}}

\renewcommand\a\alpha
\renewcommand\b\beta
\newcommand\g\gamma
\renewcommand\d\delta
\newcommand\D\Delta

\newcommand{\om}{\omega}

\begin{document}
    \bibliographystyle{alpha}

\title{Tensor structure on the Kazhdan-Lusztig category \\ for affine $\mathfrak{gl}(1|1)$}
\author{Thomas Creutzig, Robert McRae and Jinwei Yang}
\date{}

\maketitle

\begin{abstract}
We show that the Kazhdan-Lusztig category $KL_k$ of level-$k$ finite-length modules with highest-weight composition factors for the affine Lie superalgebra $\widehat{\mathfrak{gl}(1|1)}$  has vertex algebraic braided tensor supercategory structure, and that its full subcategory $\mathcal{O}_k^{fin}$ of objects with semisimple Cartan subalgebra actions is a tensor subcategory. We show that every simple $\widehat{\mathfrak{gl}(1|1)}$-module in $KL_k$ has a projective cover in $\cO_k^{fin}$, and we determine all fusion rules involving simple and projective objects in $\cO_k^{fin}$. Then using Knizhnik-Zamolodchikov equations, we prove that $KL_k$ and $\mathcal{O}_k^{fin}$ are rigid. As an application of the tensor supercategory structure on $\mathcal{O}_k^{fin}$, we study certain module categories for the affine Lie superalgebra $\widehat{\sl(2|1)}$ at levels $1$ and $-\frac{1}{2}$. In particular, we obtain a tensor category of $\widehat{\sl(2|1)}$-modules at level $-\frac{1}{2}$ that includes relaxed highest-weight modules and their images under spectral flow.
\end{abstract}

\tableofcontents

\section{Introduction}

Affine Lie (super)algebras and their representations play meaningful roles in various areas of both mathematics and physics. Increasingly, the representation categories of interest are neither finite nor semisimple, and such categories are expected to give rise to interesting non-semisimple topological field theories and invariants of knots and links. They are also expected to relate to quantum group module categories via non-semisimple Kazhdan-Lusztig correspondences.

We are concerned  with representation categories of the affine Lie superalgebra $\widehat{\gl(1|1)}$ at level $k$, or equivalently of the affine vertex operator superalgebra $V_k(\gl(1|1))$, whose associated knot and link invariants are expected to be Alexander-Conway polynomials. In physics, $V_k(\gl(1|1))$ is the chiral algebra of a prototypical logarithmic conformal field theory, namely the WZW theory of the supergroup $GL(1|1)$. After describing our main results, we will explain these connections in detail.

\subsection{Results}

In this paper, we study two representation categories for $\widehat{\gl(1|1)}$. The first is the Kazhdan-Lusztig category $KL_k$ of level-$k$ finite-length modules with highest-weight composition factors. Equivalently, this is the category of finite-length grading-restricted generalized $V_k(\gl(1|1))$-modules, where ``generalized'' means that these modules decompose as direct sums of generalized eigenspaces for the Virasoro algebra zero-mode $L_0$. The second category we consider is the full subcategory $\cO_k^{fin}$ of $KL_k$ consisting of modules with semisimple Cartan subalgebra actions. Using the sufficient conditions given in \cite{CY} for the existence of the logarithmic tensor categories constructed in \cite{HLZ0}-\cite{HLZ8}, we prove in Theorems \ref{KLktencat} and \ref{thm:Okfin_tens_cat}  that both $KL_k$ and $\cO_k^{fin}$ have the vertex and braided tensor supercategory structures of \cite{HLZ0}-\cite{HLZ8} (see \cite{CKM} for a description of the supercategory structure in the superalgebra generality).

There are two classes of simple objects in $KL_k$ and $\cO_k^{fin}$. The typical modules $\widehat{V}^k_{n,e}$ for $n\in\CC$, $e/k\notin\ZZ$ are irreducible Verma $\widehat{\gl(1|1)}$-modules; the label $(n,e)$ indicates Cartan subalgebra eigenvalues on a highest-weight vector. The atypical modules $\widehat{A}_{n, \ell k}^k$ for $n\in\CC$, $\ell\in\ZZ$ are then the unique irreducible quotients of the corresponding reducible Verma modules. In Section \ref{subsec:proj}, we show that all irreducible modules have projective covers in $\cO_k^{fin}$: the typical modules are their own projective covers, while each $\widehat{A}^k_{n,\ell k}$ has a length-$4$ projective cover $\widehat{P}^k_{n,\ell k}$.

%
% such that $e/k \notin \ZZ$ and $\ell \in \ZZ$. We construct projective covers for these simple objects in the category $\cO_k^{fin}$. The typical modules $\widehat{V}_{n,e}$ are projective covers of themselves, and the atypical modules $\widehat{A}_{n, \ell k}$ have projective covers $\widehat{P}^k_{n, \ell k}$.

We next determine tensor products of irreducible modules in $KL_k$ and $\cO_k^{fin}$ in Theorems \ref{thm:atyp_atyp_fusion}, \ref{thm:atyp_typ_fusion}, and \ref{thm:typ_typ_fusion}:
\begin{theorem}\label{maintheorem1}
The following are the tensor products of simple $\widehat{\gl(1|1)}$-modules:
\begin{enumerate}
\item For $n, n'\in\CC$ and $\ell, \ell'\in\ZZ$,
\begin{equation*}
 \widehat{A}^k_{n,\ell k}\boxtimes\widehat{A}^k_{n',\ell' k} \cong\widehat{A}^k_{n+n'-\varepsilon(\ell,\ell'),(\ell+\ell')k},
\end{equation*}
where the scalar $\varepsilon(\ell,\ell')$ is defined in \eqref{defofvar} and Theorem \ref{thm:atyp_atyp_fusion} below.

\item For $n,n'\in\CC$, $\ell\in\ZZ$, and $e'/k\notin\ZZ$,
\begin{equation*}
 \widehat{A}^k_{n,\ell k}\boxtimes\widehat{V}^k_{n',e'}\cong\widehat{V}^k_{n+n'-\varepsilon(\ell),e'+\ell k},
\end{equation*}
where the scalar $\varepsilon(\ell)$ is defined in \eqref{defofvar} below.

\item For $n,n'\in\CC$ and $e/k, e'/k\notin\ZZ$,
 \begin{equation*}
  \widehat{V}^k_{n,e}\boxtimes\widehat{V}^k_{n',e'}\cong\left\lbrace\begin{array}{lll}
  \widehat{V}_{n+n'+\frac{1}{2},e+e'}^k\oplus\widehat{V}_{n+n'-\frac{1}{2},e+e'}^k & \text{if} & (e+e')/k\notin\ZZ \\
         \widehat{P}^k_{ n+n'+\varepsilon((e+e')/k),e+e'} & \text{if} & (e+e')/k\in\ZZ                                               \\                                                   \end{array}
\right. .
\end{equation*}

% \item[(1)] $\widehat{A}_{n,0}^k\btimes \widehat{A}_{n',0}^k = \widehat{A}_{n+n',0}^k$,
% \item[(2)] $\widehat{V}_{n,e}^k \btimes \widehat{V}_{n', -e}^k = \widehat{P}_{n+n',0}^k$ for $e \neq 0$
% \item[(3)] $\widehat{A}_{n,0}^k \btimes \widehat{V}_{n', e'}^k = \widehat{V}_{n+n', e'}^k$ for $e'/k \notin \ZZ$,
% \item[(4)] $\widehat{V}_{n, e}^k \btimes \widehat{V}_{n', e'}^k = \widehat{V}_{n+n'+\frac{1}{2}, e+e'}^k \oplus \widehat{V}_{n+n'-\frac{1}{2}, e+e'}^k$ for $e'/k, (e+e')/k \notin \ZZ$,
% \item[(5)] $\widehat{A}_{n, \ell k}^k \btimes \widehat{A}_{n', \ell' k}^k = \widehat{A}_{n+n'-\varepsilon(\ell, \ell'), (\ell+\ell')k}^k$ for $\ell, \ell' \in \ZZ$,
% \item[(6)] $\widehat{A}_{n, \ell k}^k \btimes \widehat{V}_{n', e'}^k = \widehat{V}_{n+n'-\varepsilon(\ell), e'+\ell k}^k$ for $\ell \in \ZZ$, $e'/k \notin \ZZ$,
% \item[(7)] $\widehat{V}_{n, e}^k \btimes \widehat{V}_{n', e'}^k = \widehat{P}_{n+n'+\varepsilon((e+e')/k), e+e'}$ for $e/k, e'/k \notin \ZZ$ and $(e+e')/k \in \ZZ$
\end{enumerate}
\end{theorem}

These fusion rules follow from the relationship (developed in \cite{FZ1, Li, FZ2, HY, MY}, among other references) between vertex algebraic intertwining operators among $\widehat{\gl(1|1)}$-modules and $\gl(1|1)$-homomorphisms between lowest-conformal-weight spaces of $\widehat{\gl(1|1)}$-modules. In particular, if $W_1$ and $W_2$ are two simple modules in $KL_k$, the canonical tensor product intertwining operator of type $\binom{W_1\boxtimes W_2}{W_1\,W_2}$ restricts to a $\gl(1|1)$-homomorphism onto the lowest-conformal-weight space of $W_1\boxtimes W_2$, which can then be induced to a  homomorphism from a generalized Verma $\widehat{\gl(1|1)}$-module into $W_1\boxtimes W_2$. On the other hand, we can get homomorphisms coming out of $W_1\boxtimes W_2$ if we can construct intertwining operators using suitable $\gl(1|1)$-module homomorphisms. Such considerations provide enough information to determine the tensor product modules in the above theorem; for tensor products involving the atypical irreducibles $\widehat{A}^k_{n,\ell k}$, $\ell\in\ZZ\setminus\lbrace 0\rbrace$, we also use information coming from explicit singular vectors in the reducible Verma modules $\widehat{V}^k_{0,\pm k}$.

%
% The fusion rules $(1)$--$(4)$ follow from the method of determining the fusion rules developed in \cite{FZ1}, \cite{Li}, \cite{FZ2} and \cite{HY}. The fusion rule $(5)$ follows by applying the tensor structure results of orbifold vertex algebra developed in \cite{CKLR} and \cite{Mc2} to the $U(1)$-orbifold of the tensor product vertex superalgebra of $\beta\gamma$ and $bc$-systems. As a consequence of the fusion rule $(5)$, the atypical modules are all simple currents. Using this fact and \cite[Prop.~2.5]{CKLR}, we prove the fusion rules $(6)$ and $(7)$.

Finally, in Theorem \ref{rigidityofcategoryO}, we prove that $KL_k$ and $\cO_k^{fin}$ are rigid  with duals given by contragredient modules; since $KL_k$ and $\cO_k^{fin}$ are also braided and have a natural twist isomorphism, this means they are braided ribbon tensor supercategories. To prove this result, we first show that simple modules in $KL_k$ are rigid, and then we use \cite[Thm. 4.4.1]{CMY2} to extend rigidity to all finite-length modules. Rigidity for the atypical irreducibles is easy: the fusion rules of Theorem \ref{maintheorem1} show that $\widehat{A}^k_{n,\ell k}$ is a simple current, so evaluation and coevaluation morphisms are isomorphisms.

To prove that each typical irreducible module $\widehat{V}^k_{n,e}$ is rigid, with contragredient dual $\widehat{V}^k_{-n,-e}$, we need to use explicit formulas for $4$-point correlation functions of the form
\begin{equation*}
 \phi(z)=\langle v_0,\cY_1(v_1,1)\cY_2(v_2, z)v_3\rangle,
\end{equation*}
where $\cY_1$ and $\cY_2$ are suitable intertwining operators involving $\widehat{V}^k_{n,e}$ and its contragredient, $v_1$ and $v_3$ are certain lowest-conformal-weight vectors in $\widehat{V}^k_{n,e}$, and $v_0, v_2$ are certain lowest-conformal-weight vectors in $\widehat{V}^k_{-n,-e}$. We use Knizhnik-Zamolodchikov equations to show in Theorem \ref{thm:diff_eq} that $\phi(z)$ satisfies the second-order regular-singular-point differential equation
\begin{align*}
 z(1-z)\phi''(z) &+\big[(4\Delta_{n,e}+1)-(8\Delta_{n,e}+1)z\big]\phi'(z)+4\Delta_{n,e}^2 z^{-1}\phi(z)\nonumber\\
 & +2\Delta_{n,e}(2\Delta_{n,e}-1)(1-z)^{-1}\phi(z)+\left[\left(\frac{e}{k}\right)^2-16\Delta_{n,e}^2\right]\phi(z)=0,
\end{align*}
where $\Delta_{n,e}=\frac{e}{k}\left(n+\frac{e}{2k}\right)$ is the lowest conformal weight of $\widehat{V}^k_{n,e}$. This differential equation can be solved explicitly in terms of hypergeometric functions, and we can then prove rigidity for $\widehat{V}^k_{n,e}$ with the help of well-known formulas relating series expansions of hypergeometric functions on different regions of $\CC\setminus\lbrace 0,1\rbrace$.

\subsection{On relaxed highest-weight modules}

The key result we use to establish the existence of vertex tensor category structure is \cite[Thm.~3.3.4]{CY}, which shows that $KL_k$ has the vertex algebraic braided tensor (super)category structure of \cite{HLZ0}-\cite{HLZ8} provided that every $C_1$-cofinite grading-restricted generalized $V_k(\gl(1|1))$-module has finite length. A similar result was proved earlier in \cite{CHY} for affine vertex operator algebras at admissible levels, plus rigidity in the simply-laced case \cite{C}.

However, the category of $C_1$-cofinite modules for affine vertex operator algebras at admissible level is rather small. Especially, it misses all relaxed highest-weight modules and their images under spectral flow: these modules have infinite-dimensional conformal weight spaces, and conformal weights need not be lower-bounded. This means that finiteness conditions assumed in \cite{HLZ0}-\cite{HLZ8} to prove existence of vertex tensor category structure do not hold, and one thus needs other strategies to obtain tensor categories that include relaxed highest-weight modules.

For example, a vertex operator (super)algebra $V$ might have a vertex operator subalgebra $U$ that has a module category $\mathcal C_U$ that admits vertex tensor category structure. If $V$ is an object in  a suitable completion of $\mathcal C_U$,
then one can use the theory of vertex operator (super)algebra extensions \cite{HKL, CKM} to obtain and study tensor categories of $V$-modules that lie in the completion of
$\mathcal C_U$.  For this reason, we have established that under reasonable assumptions, the completion of a vertex tensor category under direct limits inherits vertex tensor category structure \cite{CMY1}. In \cite{CMY2}, we applied this strategy to module categories for singlet vertex operator algebras that live in the direct limit completions of the $C_1$-cofinite module categories for the corresponding Virasoro vertex operator subalgebras. These Virasoro module categories have vertex tensor category structure by \cite{CJORY}.

%
% modules \cite{CMY2} using that the singlet algebra is an object in the direct limit completion of the category of $C_1$-cofinite modules of the Virasoro algebra at corresponding central charge. Moreover the latter has vertex tensor category structure \cite{CJORY}.

In the present work, we apply the same idea to obtain vertex tensor categories of relaxed highest-weight modules for the simple affine vertex operator superalgebra of $\mathfrak{sl}(2|1)$ at level $-\frac{1}{2}$. The superalgebra $V_{-\frac{1}{2}}(\sl(2|1))$  is an infinite-order simple current extension of $V_1(\gl(1|1))$ \cite{CR1}, and this means that simple $V_{-\frac{1}{2}}(\sl(2|1))$-modules are certain infinite direct sums of $V_1(\gl(1|1))$-modules. As it turns out, we obtain simple $V_{-\frac{1}{2}}(\sl(2|1))$-modules with infinite-dimensional conformal weight spaces, so these are relaxed highest-weight modules. We have the following results:
\begin{enumerate}
\item Proposition \ref{prop:simple} characterizes all simple $V_{-\frac{1}{2}}(\sl(2|1))$-modules that are objects in the direct limit completion of $\cO_1^{fin}$ as inductions of simple $V_1(\gl(1|1))$-modules.
\item Equation \ref{eq:RHWch} illustrates the spectral-flow-twisted relaxed highest-weight property using the characters of the $V_{-\frac{1}{2}}(\sl(2|1))$-modules induced from typical $V_1(\gl(1|1))$-modules.
\item Proposition \ref{prop:proj} states that if a simple $V_1(\gl(1|1))$-module $S$ induces to a local $V_{-\frac{1}{2}}(\sl(2|1))$-module, then the induction of the projective cover $P_S$ of $S$  is the projective cover of the induction of $S$.
\item The induction functor is monoidal \cite[Sec. 2.7]{CKM} and preserves duals \cite[Sec. 2.8]{CKM}, so fusion rules and rigidity are inherited from  $\cO_1^{fin}$.
\end{enumerate}

Another example of a vertex operator superalgebra extension of $V_1(\gl(1|1))$ is a pair of $\beta\gamma$-ghosts tensored with two free fermions, which is just an additional order-$2$ simple current extension of $V_{-\frac{1}{2}}(\sl(2|1))$. The representation theory of $\beta\gamma$-ghosts was the subject of recent work by Allen and Wood \cite{AW}; in complete analogy to $V_{-\frac{1}{2}}(\sl(2|1))$, rigid vertex tensor category structure is  inherited from our tensor category for $V_1(\gl(1|1))$.

A third simple current extension of $V_1(\gl(1|1))$ yields $V_{1}(\sl(2|1))$, and we also study its module category obtained from  $\cO_1^{fin}$ in the last section of this work.
In this case, all $V_{1}(\sl(2|1))$-modules are ordinary, with finite-dimensional conformal weight spaces. This vertex superalgebra is however interesting from a number-theoretical point of view, as characters of atypical modules are mock modular forms \cite{BO}. If one includes Jacobi forms, then they are even mock Jacobi forms \cite{AC}. Note that the subject of \cite{AC} is a Verlinde formula associated to the mock modularity of various extensions of $V_1(\gl(1|1))$ that include $V_{1}(\sl(2|1))$.

\subsection{The WZW theory of $GL(1|1)$ and topological invariants}

In the early 1990s Rozansky and Saleur studied the Wess-Zumino-Witten theory of the Lie supergroup $GL(1|1)$ \cite{RS, RS3, RS2} in order to obtain invariants of $3$-manifolds and links. They were motivated by Witten's celebrated insight \cite{W} that the Jones polynomial arises from $SU(2)$ Chern-Simons theory, whose Hilbert space can be identified with the space of conformal blocks of the WZW theory of $SU(2)$ \cite{G}. In the $GL(1|1)$ analogue, Rozansky and Saleur  obtained Alexander-Conway  polynomials. Recall that any modular tensor category (that is, a non-degenerate semisimple finite braided ribbon tensor category) gives rise to invariants of compact $3$-manifolds \cite{Tu}. By now it is understood that non-semisimple non-finite categories can also give rise to $3$-manifold invariants via non-semisimple topological field theories (see for example \cite{CGP}). It is thus natural to expect that the invariants of Rozansky and Saleur can be reproduced from a topological field theory constructed from the tensor category $\cO_k^{fin}$ of modules for affine $\gl(1|1)$.

% they aimed was to obtain invariants of 3-manifolds and links therein. . The space of conformal blocks of the WZW theory of $SU(2)$ on the other hand can be identified with the Hilbert space of the Chern Simons theory .

The WZW theory of $GL(1|1)$ was actually the first example of a logarithmic conformal field theory to be studied in detail. The term ``logarithmic'' here refers to logarithmic singularities in the correlation functions. Such singularities arise from non-semisimple action of the Virasoro zero-mode, and by now conformal field theories associated to non-semisimple module categories for vertex operator algebras are called logarithmic conformal field theories; see \cite{CR3} for an introduction.
The $GL(1|1)$ WZW theory has been further explored in the bulk  \cite{SS, CRo}  and boundary \cite{CQS, CS}. In particular, fusion rules were suggested by computations of correlation functions \cite{SS, CRo}, boundary states \cite{CQS}, the Verlinde formula \cite{CQS, CR2}, and the NGK algorithm \cite{CR2}. Our work here shows that these methods indeed predicted the correct fusion rules.

\subsection{Outlook}

The most popular vertex operator algebras with non-semisimple representation theory are probably the $(1, p)$ singlet and triplet algebras. In the first case $p=2$, the $(1, 2)$ singlet algebra is the Heisenberg coset of  $V_k(\gl(1|1))$, while the $(1, 2)$ triplet algebra is a simple current extension of the singlet. It is also the even subalgebra of a pair of symplectic fermions, which is the affine vertex operator superalgebra of $\mathfrak{psl}(1|1)$. The representation theory of the triplet algebra was first investigated by Feigin, Gainutdinov, Tipunin and Semikhatov \cite{FGST}: they explored the conjectural Kazhdan-Lusztig correspondence with the restricted quantum group of $\sl(2)$. This conjecture is correct on the level of linear categories \cite{NT},  and fusion rules for simple and projective modules also coincide \cite{TW, KS}. However the category of the quantum group is not braidable \cite{KS}, and it has only recently been realized that there is a quasi-Hopf algebra modification of the quantum group yielding braided tensor category structure \cite{CGR}. This quasi-Hopf modification is obtained by relating the restricted quantum group to the category of local modules for a simple current extension in the category of weight modules for the unrolled restricted quantum group. The latter category is conjecturally equivalent to the category of ordinary modules for the singlet vertex operator algebra \cite{CGP2, CMR}.

Thus clearly a central problem in this area is to prove the Kazhdan-Lusztig-type correspondences between triplet algebra categories and quasi-Hopf modifications of the restricted quantum groups. A variant of this conjecture is an equivalence of our tensor category $\cO_k^{fin}$ with a category of weight modules for $U_q(\gl(1|1))$. Very recently, Babichenko studied the connection between $U_q(\gl(1|1))$ and braided tensor categories associated to solutions of the Knizhnik-Zamolodchikov equation \cite{Ba}. As the original Kazhdan-Lusztig correspondence of \cite{KL1}-\cite{KL5} is a braided equivalence between categories of modules for an affine Lie algebra and for a corresponding quantum group, where the quantum group parameter $q$ is related to the level of the affine Lie algebra, it seems realistic that similar arguments can prove a correspondence between our $\cO_k^{fin}$ and a category of weight modules for $U_q(\gl(1|1))$. Using the relation between $V_k(\gl(1|1))$ and the $(1, 2)$ singlet and triplet algebras, the quantum group triplet and singlet correspondences should then follow.

The next natural question is whether one can understand higher-rank affine superalgebras and $W$-superalgebras. As a first step, one can consider simple current extensions of tensor products of $V_k(\gl(1|1))$, for example $V_1(\gl(n|n))$  as an extension of $n$ copies of $V_k(\gl(1|1))$. Here, one first needs to show that the category of ordinary modules for the tensor product of $V_k(\gl(1|1))$ with itself has vertex tensor category structure, that is, one needs an improvement of \cite[Thm. 5.2]{CKM2} that applies to non-semisimple module categories.

Much more generally, our results can be viewed as the simplest examples of exploiting trialities of $W$-superalgebras for understanding braided tensor categories. Namely, Feigin-Semikhatov triality \cite{FS} asserts that Heisenberg cosets of subregular $W$-algebras of $\sl(n)$ coincide with Heisenberg cosets of principal $W$-superalgebras of $\sl(n|1)$, and also with cosets by the even affine subalgebra of the affine vertex superalgebra of $\sl(n|1)$. These conjectures are proven in \cite{CGN, CL}, and there also the relation of the levels is stated precisely.  The $n=1$ version of this duality asserts that the Heisenberg coset of a pair of $\beta\gamma$-ghosts coincides with the $U(1)$-orbifold of the affine vertex superalgebra of $\mathfrak{psl}(1|1)$, also known as the symplectic fermion orbifold. Moreover, these conjectures can be improved to a Kazama-Suzuki-type duality that states that a Heisenberg coset of the subregular $W$-algebra tensored with a pair of fermions is the principal $W$-superalgebra, and conversely, a Heisenberg coset of the principal $W$-superalgebra tensored with a lattice vertex superalgebra is the subregular $W$-algebra \cite{CGN}. Here, the $n=1$ version is the relation between $\beta\gamma$-ghosts and $V_k(\gl(1|1))$; this duality was recently used in \cite{AP} to compute fusion rules.

 The $n=2$ version of this duality is the Kazama-Suzuki duality between the affine vertex algebra of $\sl(2)$ and the $N=2$ super Virasoro algebra; this was exploited by Feigin, Semikhatov, and Tipunin a while ago \cite{FST} and recently has received renewed attention \cite{CLRW, KoSa, Sa1, Sa2}.
All these works use the duality for understanding super Virasoro modules. In our opinion it is easier to prove vertex tensor category structure for $V_k(\gl(1|1))$ than for $\beta\gamma$-ghosts, so we hope that one can also establish and study vertex tensor category structure on certain module categories for the super Virasoro algebra, and then use the duality to understand dual categories of relaxed highest-weight modules of the affine vertex algebra of $\sl(2)$.

\vspace{5mm}

\noindent {\bf Acknowledgements}

\noindent TC acknowledges support from NSERC discovery grant RES0048511.

\section{Tensor supercategories of \texorpdfstring{$\widehat{\gl(1|1)}$}{gl(1|1)-hat}-modules}

In this section, we first review the basic representation theory of the Lie superalgebra $\gl(1|1)$ and its affinization $\widehat{\gl(1|1)}$, including the definitions of the categories $KL_k$ and $\cO_k^{fin}$ of $\widehat{\gl(1|1)}$-modules. We then show that $KL_k$ and $\cO_k^{fin}$ have vertex algebraic braided tensor supercategory structure, and we construct projective covers of irreducible modules in $\cO_k^{fin}$.

\subsection{Representations of \texorpdfstring{$\gl(1|1)$}{gl(1|1)}}\label{sec:gl(1|1)_reps}
The general linear superalgebra $\gl(1|1)$ consists of endomorphisms of the vector superspace $\CC^{1|1}$. It has basis
\[
N = \frac{1}{2}\begin{pmatrix}
1 &  0\\
0 & -1
\end{pmatrix},\;\;\;
E = \begin{pmatrix}
1 &  0\\
0 & 1
\end{pmatrix},\;\;\;
\psi^+ = \begin{pmatrix}
0 &  1\\
0 & 0
\end{pmatrix},\;\;\;
\psi^- = \begin{pmatrix}
0 &  0\\
1 & 0
\end{pmatrix},
\]
so that the even and odd subspaces of $\gl(1|1)$ are $\gl(1|1)_{\bar{0}} = \Span\{N, E\}$ and $\gl(1|1)_{\bar{1}} = \Span\{\psi^+, \psi^-\}$, respectively. The non-zero Lie superbrackets of basis elements are
\[
[N, \psi^{\pm}] = \pm \psi^{\pm},\;\;\; \{\psi^+, \psi^-\} = E.
\]
There is a nondegenerate even invariant supersymmetric bilinear form $\kappa(\cdot, \cdot)$ on $\gl(1|1)$ such that
\[
\kappa(N, E) = \kappa(E, N) = 1,\;\;\; \kappa(\psi^+, \psi^-) = -\kappa(\psi^-, \psi^+) = 1,
\]
with $\kappa$ vanishing on all other pairs of basis elements.
\begin{remark}\label{rem:bilinear}
There is a second invariant bilinear form $\kappa_2$ on $\gl(1|1)$ that vanishes on all pairs of basis elements except for $\kappa_2(N, N) = 1$. Moreover, $\gl(1|1)$ has automorphisms $\omega_{\lambda, \mu}$ for $\lambda \in \mathbb C$, $\mu \in \mathbb C \setminus \{0\}$ defined by
\begin{equation}\nonumber
\begin{split}
\omega_{\lambda, \mu}(N) = N + \lambda E, \qquad \omega_{\lambda, \mu}(\psi^\pm) = \mu\psi^\pm, \qquad \omega_{\lambda, \mu}(E) = \mu^2 E,
\end{split}
\end{equation}
so that every non-degenerate bilinear form of $\gl(1|1)$ is related to $\kappa$ by an automorphism. That is,
\[
\kappa(\omega_{\lambda, \mu}(a), \omega_{\lambda, \mu}(b)) = \mu^2 \kappa(a, b) + 2\lambda \kappa_2 (a, b)
\]
for any $a, b \in  \gl(1|1)$.
\end{remark}
Following the notation of \cite{CR2}, let $V_{n-\frac{1}{2},e}$ for $n,e\in\CC$ be the Verma module generated by a highest-weight vector $v$ such that
\[
N\cdot v = nv,\;\;\; E\cdot v = ev,\;\;\;\psi^+\cdot v=0.
\]
Since $\psi^{-}$ squares to zero in the universal enveloping algebra of $\gl(1|1)$, every Verma module has dimension $2$; thus $n$ is the average of the two $N$-eigenvalues of $V_{n,e}$.  The Verma module $V_{n, e}$ is irreducible if and only if $e \neq 0$. When $e = 0$, we denote the $1$-dimensional irreducible quotient of $V_{n,e}$ by $A_{n+\frac{1}{2}}$. The irreducibles with $e\neq 0$ are said to be {\em typical} and those with $e=0$ are said to be {\em atypical}. For each $n\in\CC$, there is a non-split exact sequence
\[
0 \rightarrow A_{n-\frac{1}{2}} \rightarrow V_{n,0} \rightarrow A_{n+\frac{1}{2}} \rightarrow 0.
\]
For $n\in\CC$, we also define the induced module $P_n=\cU(\gl(1|1))\otimes_{\cU(\gl(1|1)_{\bar{0}})} \CC v_n$, where $E\cdot v_n=0$ and $N\cdot v_n=nv_n$ (that is, $\CC v_n$ is the restriction of $A_n$ to $\gl(1|1)_{\bar{0}}$). The module $P_n$ is indecomposable but reducible and satisfies the non-split exact sequence
\begin{equation}\label{seq:Pn_ext}
0 \rightarrow V_{n+\frac{1}{2},0} \rightarrow P_n \rightarrow V_{n-\frac{1}{2},0} \rightarrow 0.
\end{equation}
It has Loewy diagram
\[
\begin{tikzpicture}[->,>=latex,scale=1.5]
\node (b1) at (1,0) {$A_{n}$};
\node (c1) at (-1.2, 1){$P_{n}$:};
   \node (a1) at (0,1) {$A_{n-1}$};
   \node (b2) at (2,1) {$A_{n+1}$};
    \node (a2) at (1,2) {$A_{n}$};
\draw[] (b1) -- node[left] {} (a1);
   \draw[] (b1) -- node[left] {} (b2);
    \draw[] (a1) -- node[left] {} (a2);
    \draw[] (b2) -- node[left] {} (a2);
\end{tikzpicture}
\]
and has basis $\{v_n, \psi^+ v_n, \psi^- v_n, \psi^+\psi^- v_n\}$.

Let $\cO$ be the category of finitely-generated $\gl(1|1)$-modules with semisimple $\gl(1|1)_{\bar{0}}$-actions (they have nilpotent $\gl(1|1)_{\bar{1}}$-actions automatically); see for example the exposition in \cite{B} for more details on this category. Every object in $\cO$ has finite length, and $\cO$ has enough projective objects. The typical irreducible modules $V_{n,e}$ for $n\in\CC$, $e\in\CC\setminus\lbrace 0\rbrace$ are their own projective covers, while $P_n$ for $n\in\CC$ is the projective cover of the atypical irreducible module $A_n$.

% The projective covers of the simple modules are as follows:
% \begin{lemma}
% For $n \in \RR$, $e \in \RR \backslash \{0\}$, the objects $V_{n,e}$ are projective covers of $V_{n,e}$ themselves and $P_n$ are projective covers of $A_n$ in the category $\cO$.
% \end{lemma}

We remark that $\cO$ is more precisely a \textit{supercategory}: every object $M$ has a $\ZZ_2$-grading $M_{\bar{0}}\oplus M_{\bar{1}}$ such that the $\gl(1\vert 1)_{\bar{0}}$-action is even. The $\ZZ_2$-gradings on modules induce $\ZZ_2$-gradings on morphism spaces: a linear map $f: M\rightarrow N$ is a morphism of $\gl(1|1)$-modules of parity $\vert f\vert$ if
\begin{equation*}
 f(M_i)\subseteq N_{i+\vert f\vert}
\end{equation*}
for $i\in\ZZ_2$ and
\begin{equation*}
 a\cdot f(m) =(-1)^{\vert a\vert\vert f\vert} f(a\cdot m)
\end{equation*}
for $m\in M$ and homogeneous $a\in\gl(1|1)$. The Verma modules $V_{n,e}$, for example, have $\ZZ_2$-gradings such that $v$ is even and $\psi^- v$ is odd. We can reverse parities to obtain a different object $\Pi(V_{n,e})$ which is isomorphic to $V_{n,e}$ via an odd isomorphism.

\subsection{Representations of the affine Lie superalgebra $\widehat{\gl(1|1)}$}
The affine Lie superalgebra $\widehat{\gl(1|1)}$ associated with $\gl(1|1)$ and the bilinear form $\kappa(\cdot, \cdot)$ is the superspace $\gl(1|1)\otimes \CC[t,t^{-1}] \oplus \CC{\bf k}$, where $\CC[t,t^{-1}]$ and $\mathbf{k}$ are even, with bracket defined by
\[
[a \otimes t^r, b \otimes t^s] = [a, b]\otimes t^{r+s} + \kappa(a, b)r\delta_{r+s,0}{\bf k},
\]
\[
[a \otimes t^r, {\bf k}] = 0
\]
for $a, b \in \gl(1|1)$ and $r, s \in \ZZ$. The non-vanishing brackets of basis elements for $\widehat{\gl(1|1)}$ are
\[
[N_r, E_s] = r{\bf k}\delta_{r+s,0}, \;\;\; [N_r, \psi^{\pm}_s] = \pm \psi^{\pm}_{r+s}, \;\;\; \{\psi^+_r, \psi^-_s\} = E_{r+s}+r{\bf k}\delta_{r+s,0},
\]
where $a_r$ denotes $a\otimes t^r$ for $a \in \gl(1|1)$ and $r \in \ZZ$.

Given a $\gl(1|1)$-module $M$ and $k \in \CC$,  $M$ is a $\gl(1|1)\otimes \CC[t]\oplus \CC{\bf k}$-module with $\gl(1|1)\otimes t\CC[t]$ acting trivially and ${\bf k}$ acting as scalar multiplication by $k$. We then have the induced $\widehat{\gl(1|1)}$-module
\[
\widehat{M}^k = \cU(\widehat{\gl(1|1)})\otimes_{\cU(\gl(1|1)\otimes \CC[t]\oplus \CC{\bf k})}M.
\]
When $M$ is the trivial $\gl(1|1)$-module $A_0=\CC \one$, $\widehat{M}^k$ has a vertex superalgebra structure, which we denote by $V_k(\gl(1|1))$. For general $M$, the modules $\widehat{M}^k$ are also $V_k(\gl(1|1))$-modules, and $k$ is called the level of $\widehat{M}^k$.
\begin{remark}
The space of invariant bilinear forms of $\gl(1|1)$ is two-dimensional. \textit{A priori}, this means that one should consider a two-parameter family of affine vertex superalgebras $V_B(\gl(1|1))$ associated to $\gl(1|1)$, parameterized by bilinear forms $B$. As noted in Remark \ref{rem:bilinear} all non-degenerate bilinear forms are related by an automorphism of $\gl(1|1)$, and hence $V_B(\gl(1|1)) \cong V_1(\gl(1|1))$ if $B$ is non-degenerate. It thus suffices to restrict attention to $V_1(\gl(1|1))$. Since there is no additional complication for general non-zero $k$, we will stay with that notation.
\end{remark}

When $k\neq 0$, $V_k(\gl(1|1))$ is a vertex operator superalgebra with conformal element constructed in \cite{RS} using a modified Sugawara construction:
\begin{equation}
\om = \frac{1}{2k}(N_{-1}E_{-1}+E_{-1}N_{-1}-\psi^+_{-1}\psi^-_{-1}+\psi^-_{-1}\psi^+_{-1})\one + \frac{1}{2k^2}E_{-1}^2\one.
\end{equation}
Its associated vertex operator is given by
\begin{equation}
Y(\om, x) = \frac{1}{2k}:N(x)E(x)+E(x)N(x)-\psi^+(x)\psi^-(x)+\psi^-(x)\psi^+(x):+\frac{1}{2k^2}:E(x)E(x):.
\end{equation}
In particular,
\begin{align}
\label{ell-1}L_{-1} & = \frac{1}{k}\sum_{r\in \ZZ_{\geq 0}}(N_{-r-1}E_r +E_{-r-1}N_r-\psi_{-r-1}^+\psi_{r}^{-}+\psi_{-r-1}^{-}\psi_r^+)+\frac{1}{k^2}\sum_{r\in \ZZ_{\geq 0}}E_{-r-1}E_{r}, \\
\label{ell-0}L_{0} & = \frac{1}{k}\sum_{r\in \ZZ_{> 0}}(N_{-r}E_r +E_{-r}N_r-\psi_{-r}^+\psi_{r}^{-}+\psi_{-r}^{-}\psi_r^+)+\frac{1}{k^2}\sum_{r\in \ZZ_{> 0}}E_{-r}E_{r}\nonumber\\
& \qquad + \frac{1}{k}(N_0E_0 - \psi_{0}^+\psi_{0}^{-}) + \frac{1}{2k}E_0 + \frac{1}{2k^2}E_0^2.
\end{align}

Now we introduce two representation categories of $V_k(\gl(1|1))$ of main interest to us. For a precise definition of the supercategory of (grading-restricted generalized) modules for a vertex operator superalgebra, see \cite[Def. 3.1]{CKM}.
\begin{defn} ${}$
\begin{enumerate}
\item The Kazhdan-Lusztig category $KL_k$ is the supercategory of finite-length grading-restricted generalized $V_k(\gl(1|1))$-modules.

\item The supercategory $\cO_k^{fin}$ is the full subcategory of $KL_k$ consisting of modules on which $N_0$ and $E_0$ act semisimply.
\end{enumerate}
\end{defn}

We start by describing the simple objects in $KL_k$. First, just as for affine vertex operator algebras (see for example \cite[Sec. 6.2]{LL}), every irreducible (grading-restricted) $V_k(\gl(1|1))$-module is generated by its lowest conformal weight space, which must be an irreducible finite-dimensional $\gl(1|1)$-module. Thus every simple module in $KL_k$ is a quotient of a generalized Verma $\widehat{\gl(1|1)}$-module $\widehat{V}^k_{n,e}$ for $e\neq 0$ or $\widehat{A}^k_{n,0}$. The structure of these generalized Verma $\widehat{\gl(1|1)}$-modules was determined in \cite[Sec. 3.2]{CR2}. In particular,
\begin{enumerate}
 \item  The generalized Verma module $\widehat{V}_{n, e}^k$ is irreducible if and only if $e/k\notin\ZZ$.

 \item When $e=0$, the unique irreducible quotient of $\widehat{V}_{n,0}^k$ is $\widehat{A}_{n+\frac{1}{2},0}^k$, and there is a non-split exact sequence
 \begin{equation}\label{extatypical}
0 \rightarrow \widehat{A}_{n-\frac{1}{2}, 0}^k \rightarrow \widehat{V}_{n,0}^k \rightarrow \widehat{A}_{n+\frac{1}{2},0}^k \rightarrow 0.
\end{equation}

\item When $e/k\in\ZZ\setminus\lbrace 0\rbrace$, we use $\widehat{A}_{n,e}^k$ to denote the unique irreducible quotient of $\widehat{V}_{n,e}^k$, and there are non-split exact sequences
\begin{eqnarray}\label{exactsequencesimple}
&& 0 \rightarrow \widehat{A}_{n+1, e}^k \rightarrow \widehat{V}_{n,e}^k \rightarrow \widehat{A}_{n,e}^k \rightarrow 0\;\;\; (e/k = 1, 2, 3, \dots),\nonumber\\
&& 0 \rightarrow \widehat{A}_{n-1, e}^k \rightarrow \widehat{V}_{n,e}^k \rightarrow \widehat{A}_{n,e}^k \rightarrow 0\;\;\; (e/k = -1, -2, -3, \dots).
\end{eqnarray}
\end{enumerate}
Thus $KL_k$ has two classes of simple objects: the \textit{typical} irreducible modules $\widehat{V}_{n,e}^k$ for $n\in\CC$, $e/k\notin\ZZ$, and the \textit{atypical} irreducible modules $\widehat{A}_{n,e}$ for $n\in\CC$, $e/k\in\ZZ$. Note that as a module for itself, $V_k(\gl(1|1))$ is isomorphic to $\widehat{A}_{0,0}^k$; in particular, $V_k(\gl(1|1))$ is a simple vertex operator superalgebra. All simple modules in $KL_k$ are $C_1$-cofinite,
and they are also objects of $\cO_k^{fin}$.

We now discuss some properties of $\cO_k^{fin}$ that we will use in the following sections. First, since $E_0$ is central in $\widehat{\gl(1|1)}$ and acts semisimply on modules in $\cO_k^{fin}$, we have a direct sum decomposition
\[
\cO_k^{fin} = \bigoplus_{e \in \CC}(\cO_k^{fin})_{e}
\]
where $(\cO_k^{fin})_e$ is the full subcategory of modules in $\cO_k^{fin}$ on which $E_0$ acts by the scalar multiplication $e$. In particular, there are no non-zero morphisms between modules in $(\cO_k^{fin})_{e_1}$ and $(\cO_k^{fin})_{e_2}$ when $e_1\neq e_2$.

Now suppose $W$ is a module in $(\cO_k^{fin})_e$ for some $e\in\CC$. If $v_{n,e}\in W$ is some highest-weight vector for $\widehat{\gl(1|1)}$ with $N_0$-eigenvalue $n+\frac{1}{2}$ (in particular $\psi^+_0\cdot v_{n,e}=0$), then \eqref{ell-0} shows that $v_{n,e}$ is an $L_0$-eigenvector with conformal weight
\begin{equation}\label{eqn:Delta_n,e}
 \Delta_{n,e}=\frac{e}{k}\left(n+\frac{e}{2k}\right).
\end{equation}
That is, the minimal conformal weights of (grading-restricted) modules in $(\mathcal{O}_k^{fin})_e$ have the form $\Delta_{n,e}$ for $n\in\CC$; more precisely, the conformal weights of such modules lie in $\cup_i (\Delta_{n_i,e}+\NN)$ for finitely many $n_i\in\CC$. In the case $e=0$, this means:
\begin{lemma}\label{lem:usefullemma_1}
 For any non-zero module $W$ in $(\cO_k^{fin})_0$:
 \begin{enumerate}
  \item The unique minimal conformal weight space of $W$ is $W_{[0]}$.
  \item $W$ is generated by $W_{[0]}$ as a $\widehat{\gl(1|1)}$-module.
 \end{enumerate}
\end{lemma}
\begin{proof}
 The first part is immediate from the $e=0$ case of \eqref{eqn:Delta_n,e}. For the second part, $W/\widehat{\gl(1|1)}\cdot W_{[0]}$ is a module in $(\cO_k^{fin})_0$ with vanishing conformal weight-$0$ space, so the first part implies $W/\widehat{\gl(1|1)}\cdot W_{[0]}=0$.
\end{proof}

Now recall that the \textit{contragredient} of a generalized module $W=\bigoplus_{h\in\CC} W_{[h]}$ for a vertex operator (super)algebra $V$ is a module structure on the graded dual $W'=\bigoplus_{h\in\CC} W_{[h]}^*$. For superalgebras, $W'$ has a $\ZZ_2$-grading given by $W'_{i} =(W_i)'$ for $i\in\ZZ_2$, and we define the vertex operator $Y_{W'}$ by
\begin{equation}\label{eqn:contra_vrtx_op}
 \langle Y_{W'}(v,x)w',w\rangle =(-1)^{\vert v\vert\vert w'\vert}\langle w', Y_W(e^{xL_1}(-x^{-2})^{L_0} v, x^{-1})w\rangle,
\end{equation}
following the convention of \cite{CKM}. (Note that for vertex operator superalgebras, there are several different but equivalent definitions in the literature for the contragredient module vertex operator; see \cite[Rem. 3.5]{CKM}.) The contragredient of an irreducible $V$-module is irreducible (see for example \cite[Prop. 5.3.2]{FHL}), and for a $V$-module $W$, there is a natural even isomorphism $\delta_W: W\rightarrow W''$ defined by
\begin{equation*}
 \langle\delta_W(w),w'\rangle =(-1)^{\vert w\vert\vert w'\vert}\langle w',w\rangle
\end{equation*}
for $w\in W$, $w'\in W'$. Moreover, taking contragredients defines an even contravariant functor on $V$-modules, with the contragredient of a parity-homogeneous morphism $f: W_1\rightarrow W_2$ defined by
\begin{equation*}
 \langle f'(w_2'), w_1\rangle =(-1)^{\vert f\vert\vert w_2'\vert}\langle w_2',f(w_1)\rangle
\end{equation*}
for $w_1\in W_1$ and parity-homogeneous $w_2'\in W_2'$. It is straightforward to show that the contragredient functor preserves exactness of sequences involving parity-homogeneous homomorphisms.

For $V_k(\gl(1|1))$, \eqref{eqn:contra_vrtx_op} implies that the actions of $\widehat{\gl(1|1)}$ and the Virasoro algebra on a contragredient module $W'$ are given by
\begin{align}\label{eqn:aff_contra}
 & \langle a_r' w', w\rangle  =-(-1)^{\vert a\vert\vert w'\vert}\langle w', a_{-r} w\rangle\\
 & \langle L'(n)w',w\rangle  =\langle w',L(-n)w\rangle\nonumber
\end{align}
for $a\in\gl(1|1)$, $r,n\in\ZZ$, $w'\in W'$, and $w\in W$. In particular, the lowest conformal weight space of the contragredient $(\widehat{M}^k)'$ of a generalized Verma module is the $\gl(1|1)$-module $M^*$. Using this observation, we can determine the contragredients of many modules in $\cO_k^{fin}$:
\begin{prop}\label{prop:contra}
 Contragredients of $V_k(\gl(1|1))$-modules are as follows:
 \begin{enumerate}
  \item $(\widehat{A}^k_{n,\ell k})'\cong\widehat{A}^k_{-n,-\ell k}$ for $n\in\CC$ and $\ell\in\ZZ$.
  \item $(\widehat{V}^k_{n,e})'\cong\widehat{V}^k_{-n,-e}$ for $n\in\CC$ and $e/k\notin\ZZ$.
  \item $(\widehat{P}^k_{n})'\cong\widehat{P}^k_{-n}$ for $n\in\CC$.
 \end{enumerate}
\end{prop}
\begin{proof}
 For $W$ an irreducible $V_k(\gl(1|1))$-module, $W'$ is an irreducible $V_k(\gl(1|1))$-module whose lowest conformal weight space is the $\gl(1|1)$-dual of the lowest conformal weight space of $W$. Thus the first two cases of the proposition follow from the identities $A_n^*\cong A_{-n}$ for $n\in\CC$ and $V_{n,e}^*\cong V_{-n,-e}$ for $e\neq 0$ (by an odd isomorphism).

 For the third case, we use the even isomorphism $P_n^*\cong P_{-n}$ in category $\cO$. By the universal property of induced $\widehat{\gl(1|1)}$-modules, this isomorphism extends to an even homomorphism $f: \widehat{P}^k_{-n}\rightarrow(\widehat{P}^k_n)'$ which is an isomorphism on lowest conformal weight spaces. Since $\ker f$ and $\coker f$ are objects of $(\cO_k^{fin})_0$ with $(\ker f)_{[0]}=0=(\coker f)_{[0]}$, Lemma \ref{lem:usefullemma_1} implies that $\ker f=0=\coker f$, so that $f$ is an isomorphism.
\end{proof}

%
%  If $e/k \in \ZZ$, then $\widehat{V}_{n, e}^k$ is reducible but indecomposable, and we denote the irreducible quotient by $\widehat{A}_{n,e}^k$. There are exact sequences
%
% When $e = 0$, we denote the irreducible quotient of $\widehat{V}_{n,0}^k$ by $\widehat{A}_{n+\frac{1}{2}, 0}^k$, and the exact sequence is

\subsection{Tensor supercategory structure}

We now establish vertex and braided tensor supercategory structure on $KL_k$ and on its subcategory $\cO_k^{fin}$. The key result we use is \cite[Thm.~3.3.4]{CY}, which shows that $KL_k$ has the vertex algebraic braided tensor (super)category structure of \cite{HLZ0}-\cite{HLZ8} provided that every lower-bounded $C_1$-cofinite $V_k(\gl(1|1))$-module is a finite-length generalized module. We remark that although the results of \cite{HLZ0}-\cite{HLZ8} and \cite{CY} on vertex tensor category structure are proved for vertex operator algebras, careful examination of the proofs shows that the results generalize to module supercategories for superalgebras. For a description of the vertex and braided tensor supercategory structures on a module supercategory for a vertex operator superalgebra, see \cite[Sec. 3.3]{CKM}.

%
% The existence of the vertex tensor category structure on $KL_k$ follows from  and that $KL_k$ is the same as the category of lower bounded $C_1$-cofinite modules:
\begin{theorem}\label{KLktencat}
The supercategory $KL_k$ has vertex algebraic braided tensor supercategory structure.
\end{theorem}
\begin{proof}
By \cite[Thm. 3.3.4]{CY}, we just need to show that every lower-bounded $C_1$-cofinite $V_k(\gl(1|1))$-module has finite length. From the proof of \cite[Thm.~3.3.5]{CY}, such a module $W$ has a finite filtration
\[
0 \subset W_1 \subset W_2 \subset \cdots \subset W_n \subset W_{n+1} = W
\]
such that each $W_i$ is a $\ZZ_2$-graded submodule of $W$ and such that each $W_{i+1}/W_i$ for $i = 1, \dots, n$ is a quotient of a Verma modules $\widehat{V}_{n,e}^k$ for some $n,e \in \CC$ (or its parity-reversed version). As the exact sequences \eqref{extatypical} and \eqref{exactsequencesimple} show that each $\widehat{V}_{n,e}^k$ has length at most $2$, $W$ has finite length.
\end{proof}

% \begin{remark}
% The existence results of vertex tensor category structures in \cite{HLZ0}--\cite{HLZ8} and in \cite{CY} are proved for appropriate categories of modules for a vertex operator algebra, but by carefully examining the proofs, we conclude that the results can be generalized to the module categories for a vertex operator superalgebra.
% \end{remark}

Now the subcategory $\cO_k^{fin}$ inherits vertex and braided tensor supercategory structure from $KL_k$, provided it is closed under the $P(z)$-tensor products of \cite{HLZ3}:
\begin{theorem}\label{thm:Okfin_tens_cat}
 The supercategory $\mathcal{O}_k^{fin}$ has vertex algebraic braided tensor supercategory structure.
\end{theorem}
\begin{proof}
 Suppose $W_1$ and $W_2$ are any two objects of $\cO_k^{fin}$ and $(W_1\boxtimes_{P(z)} W_2,\boxtimes_{P(z)})$ is their $P(z)$-tensor product in $KL_k$, with $\boxtimes_{P(z)}$ the canonical even $P(z)$-intertwining map of type $\binom{W_1\boxtimes_{P(z)} W_2}{W_1\,W_2}$. It is clear that if $W_1\boxtimes_{P(z)} W_2$ is an object of $\cO_k^{fin}$, then $(W_1\boxtimes_{P(z)} W_2,\boxtimes_{P(z)})$ also satisfies the universal property of \cite[Def. 4.15]{HLZ3} for a $P(z)$-tensor product in $\cO_k^{fin}$. Thus the vertex algebraic braided tensor supercategory structure on $KL_k$ will restrict to such structure on $\cO_k^{fin}$.

 It remains to show that $E_0$ and $N_0$ act semisimply on $W_1\boxtimes_{P(z)} W_2$. By (the superalgebra generalization of) \cite[Eqn. 4.34]{HLZ3},
 \begin{equation*}
v_m(w_1\btimes_{P(z)}w_2) = (-1)^{\vert v\vert\vert w_1\vert}w_1\btimes_{P(z)}v_mw_2 + \sum_{i\geq 0}\binom{m}{i}z^{m-i}(v_iw_1)\btimes_{P(z)} w_2
\end{equation*}
for homogeneous $v\in V_k(\gl(1|1))$, $w_1\in W_1$, $w_2\in W_2$, and $m\in\ZZ$. In particular,
\[
X_0(w_1\btimes_{P(z)}w_2) = X_0w_1\btimes_{P(z)}w_2 + w_1 \btimes_{P(z)}X_0w_2
\]
for $X=E,N$. Thus because $W_1\boxtimes_{P(z)} W_2$ is spanned by projections of vectors $w_1\boxtimes_{P(z)} w_2$ to the conformal weight spaces of $W_1\boxtimes_{P(z)} W_2$ (see \cite[Prop. 4.23]{HLZ3}) and because $E_0$ and $N_0$ act semisimply on $W_1$ and $W_2$, we see that $W_1\boxtimes_{P(z)} W_2$ is an object of $\cO_k^{fin}$.
\end{proof}
\begin{remark}\label{rem:useful}
 The proof of the preceding theorem shows more specifically that if $W_1$ is an object of the subcategory $(\cO_k^{fin})_{e_1}$ for $e_1\in\CC$ and $W_2$ is an object of $(\cO_k^{fin})_{e_2}$, then $W_1\boxtimes_{P(z)} W_2$ is an object of $(\cO_k^{fin})_{e_1+e_2}$. We shall use this fact in computing fusion rules for modules in $\cO_k^{fin}$.
\end{remark}

In the braided tensor supercategory structure on $KL_k$ and $\cO_k^{fin}$, the tensor product bifunctor is given by the $P(1)$-tensor product $\boxtimes_{P(1)}$, which for simplicity we will denote by $\boxtimes$. For any two modules $W_1$, $W_2$ in $KL_k$ or $\cO_k^{fin}$, the canonical $P(1)$-intertwining map $\boxtimes$ corresponds to an even (logarithmic) intertwining operator $\cY_\boxtimes(\cdot,x)$ of type $\binom{W_1\boxtimes W_2}{W_1\,W_2}$ such that
\begin{equation*}
 \cY_\boxtimes(w_1,1)w_2 = w_1\boxtimes w_2
\end{equation*}
for $w_1\in W_1$ and $w_2\in W_2$; here the substitution $x\mapsto 1$ is realized using the branch $\log 1=0$ of logarithm. The pair $(W_1\boxtimes W_2,\cY_\boxtimes)$ then satisfies the following universal property: for any intertwining operator $\cY$ of type $\binom{W_3}{W_1\,W_2}$ with $W_3$ an object of $KL_k$, there is a unique $V_k(\gl(1|1))$-module homomorphism $f: W_1\boxtimes W_2\rightarrow W_3$ such that $f\circ\cY_\boxtimes=\cY$.

For more details on vertex algebraic braided tensor supercategory structure, in particular the unit and associativity isomorphisms which we will use briefly in Section \ref{subsec:rigidity}, see \cite{HLZ8} or the exposition in \cite[Sec. 3.3]{CKM}.

%
%
% Let $W_1$ and $W_2$ be two objects in $KL_k$. Then for a homogeneous element $v \in V_k(\gl(1|1))$ and $w_1 \in W_1$, $w_2 \in W_2$,
%
% for $m \in \ZZ$ (see \cite[Eqn.~4.34]{HLZ3}). In particular,
% $ \in \gl(1|1)$. So if $H_0$ and $E_0$ act semisimply on $W_1$ and $W_2$, they also act semisimply on $W_1 \btimes_{P(z)} W_2$, it follows that $\cO_k^{fin}$ is closed under the $P(z)$-tensor product. Also because every element in $\cO_k^{fin}$ is grading-restricted $C_1$-cofinite, $\cO_k^{fin}$ satisfies the convergence and extension property. To prove the associativity isomorphism, we still need to prove the lower bounded submodule $W_{\lambda}$ of $(W_1\otimes W_2)^*$ generated by a generalized eigenvalue $\lambda \in (W_1\otimes W_2)^*$ for $L_{P(z)}'(0)$ is in $\cO_k^{fin}$ (see \cite[Theorem~3.1]{H6}). From the proof of \cite[Theorem~3.3.4]{CY}, $W_{\lambda}$ is an object in $KL_k$. Because $H_0$ and $E_0$ act semisimply on $(W_1\otimes W_2)^*$ (see \cite[Eqn.~5.87]{HLZ4}), they also act semisimply on its submodule $W_{\lambda}$. Thus $W_{\lambda}$ is an object in $\cO_k^{fin}$. As a result, we have proved:
% \begin{cor}
% The category $\cO_k^{fin}$ has a vertex tensor category structure.
% \end{cor}

\subsection{Projective objects}\label{subsec:proj}

In this section, we show that every irreducible $V_k(\gl(1|1))$-module has a projective cover in $\cO_k^{fin}$, beginning with the atypical irreducible modules.

We will show below that for $n\in\CC$, the generalized Verma module $\widehat{P}_n^k$ is the projective cover of $\widehat{A}^k_{n,0}$. Since induction is an exact functor between module categories for Lie superalgebras (thanks to the Poincar\'{e}-Birkhoff-Witt Theorem for superalgebras), \eqref{seq:Pn_ext} induces to the exact sequence
\begin{equation}\label{proj1}
0 \rightarrow \widehat{V}_{n+\frac{1}{2}, 0}^k \rightarrow \widehat{P}_{n}^k \rightarrow \widehat{V}_{n-\frac{1}{2}, 0}^k \rightarrow 0
\end{equation}
for $n\in\CC$.

%
% From the Loewy diagram of the $\gl(1|1)$-module $P_n$ and the fact that all the singular vectors of  $\widehat{P}_{n}^k$ lie at the top, we can easily deduce the Loewy diagram of $\widehat{P}_{n}^k$:
\begin{prop}\label{prop:Pnk_Loewy}
The Loewy diagram of $\widehat{P}_{n}^k$ is
\[
\begin{tikzpicture}[->,>=latex,scale=1.5]
\node (b1) at (1,0) {$\widehat{A}^k_{n, 0}$};
\node (c1) at (-1.2, 1){$\widehat{P}_{n}^k$:};
   \node (a1) at (0,1) {$\widehat{A}^k_{n+1, 0}$};
   \node (b2) at (2,1) {$\widehat{A}^k_{n-1, 0}.$};
    \node (a2) at (1,2) {$\widehat{A}^k_{n, 0}$};
\draw[] (b1) -- node[left] {} (a1);
   \draw[] (b1) -- node[left] {} (b2);
    \draw[] (a1) -- node[left] {} (a2);
    \draw[] (b2) -- node[left] {} (a2);
\end{tikzpicture}
\]
Moreover, $\widehat{P}_{n}^k$ is a logarithmic module with $L_0$-block size $2$.
\end{prop}
\begin{proof}
The Loewy diagram of $\widehat{P}_{n}^k$ follows from the exact sequences \eqref{extatypical}, \eqref{proj1} and the Loewy diagram of $P_n$. To prove that $\widehat{P}_{n}^k$ is logarithmic, take a generator $v_n$ of the lowest conformal weight space  $(\widehat{P}_{n}^k)_{[0]}$, which is the $\gl(1|1)$-module $P_n$ (see Sec. \ref{sec:gl(1|1)_reps}). Then from the expression \eqref{ell-0} for $L_0$, we get $L_0v_n = -\frac{1}{k}\psi^+\psi^-v_n\neq 0$, while $L_0^2v_n = 0$. So $v_n$ is a generalized eigenvector for $L_0$ with block size $2$, and since $v_n$ generates $\widehat{P}_n^k$ as a $V_k(\gl(1|1))$-module, $\widehat{P}_n^k$ is a logarithmic module with $L_0$-block size $2$.
\end{proof}

To obtain the projective covers of the atypical irreducibles $\widehat{A}_{n,\ell k}^k$, $\ell\in\ZZ\setminus\lbrace 0\rbrace$, we need the {\em spectral flow} automorphisms $\sigma^\ell$  of $\widehat{\gl(1|1)}$:
\begin{equation}\label{spectralflow}
\sigma^{\ell}(N_r) = N_r,\;\;\; \sigma^{\ell}(E_r) = E_r - \ell{\bf k}\delta_{r,0},\;\;\; \sigma^{\ell}(\psi^{\pm}_r) = \psi^{\pm}_{r\mp\ell},\;\;\;\sigma^\ell(\mathbf{k})=\mathbf{k}.
% \sigma^{\ell}(L_0) = L_0 - \ell N_0.
\end{equation}
For $\ell\in\ZZ$, $\sigma^\ell$ extends to a vertex superalgebra automorphism of $V_k(\gl(1|1))$, and for a $V_k(\gl(1|1))$-module $W$, the spectral flow module $\sigma^\ell(W)$ has the same underlying superspace of $W$ but has vertex operator
\begin{equation*}
 Y_{\sigma^\ell(W)}(v,x)=Y_W(\sigma^{-\ell}(v),x)
\end{equation*}
for $v\in V_k(\gl(1|1))$.

For $n\in\CC$ and $\ell \in -\ZZ_{+}$, we define
\[
\widehat{P}_{n-\ell-\frac{1}{2}, \ell k}^k : = \sigma^{\ell}(\widehat{P}_{n}^k).
\]
Since it was shown in \cite[Sec. 3.2]{CR2} that $\sigma^{\ell}(\widehat{V}_{n,0}^k) = \widehat{V}_{n-\ell, \ell k}^k$, applying $\sigma^{\ell}$ to \eqref{proj1} yields the exact sequence
\begin{equation}\label{proj2}
0 \rightarrow \widehat{V}_{n-\ell+\frac{1}{2}, \ell k}^k \rightarrow \widehat{P}_{n-\ell-\frac{1}{2}, \ell k}^k \rightarrow \widehat{V}_{n-\ell-\frac{1}{2}, \ell k}^k \rightarrow 0.
\end{equation}
For $\ell \in \ZZ_{+}$, we also need the \textit{conjugation} automorphism of $\widehat{\gl(1|1)}$ defined by
\[
w(N_r) = -N_r,\;\;\; w(E_r) = -E_r,\;\;\; w(\psi_r^+) = \psi_r^-,\;\;\; w(\psi_r^-) = -\psi_r^+,\;\;\;w(\mathbf{k})=\mathbf{k}.
\]
We then define
\[
\widehat{P}_{-n-\ell+\frac{1}{2}, \ell k}^k := \sigma^{\ell}(w(\widehat{P}_{n}^k))
\]
for $\ell\in\ZZ_+$. Since $\sigma^{\ell}(w(\widehat{V}_{n,0}^k)) = \widehat{V}_{-n-\ell, \ell k}^k$ (again see \cite[Sec. 3.2]{CR2}), applying $\sigma^{\ell}\circ w$ to \eqref{proj1} yields the exact sequence
\begin{equation}\label{proj3}
0 \rightarrow \widehat{V}_{-n-\ell-\frac{1}{2}, \ell k}^k \rightarrow \widehat{P}_{-n-\ell+\frac{1}{2}, \ell k}^k \rightarrow \widehat{V}_{-n-\ell+\frac{1}{2}, \ell k}^k \rightarrow 0.
\end{equation}
If we use the notation
\begin{equation}\label{defofvar}
\varepsilon(\ell) = \begin{cases} \;\;\,\frac{1}{2} & \text{if}\ \ell \in \ZZ_{+}     \\ \;\;\, 0 & \text{if}\ \ell = 0\; \\ -\frac{1}{2} & \text{if}\ \ell \in -\ZZ_+ \end{cases} ,
\end{equation}
the exact sequences (\ref{proj2}) and (\ref{proj3}) can be written uniformly:
\begin{equation}\label{proj4}
0 \rightarrow \widehat{V}_{n-\ell-\varepsilon(\ell), \ell k}^k \rightarrow \widehat{P}_{n-\ell+\varepsilon(\ell), \ell k}^k \rightarrow \widehat{V}_{n-\ell+\varepsilon(\ell), \ell k}^k \rightarrow 0
\end{equation}
for $n\in\CC$, $\ell\in\ZZ\setminus\lbrace 0\rbrace$.

We will sometimes use the notational convention $\widehat{P}_{n,0}^k$ for $\widehat{P}_{n}^k$. Now we can prove:
\begin{theorem}
For $n\in\CC$ and $\ell\in\ZZ$, the modules $\widehat{P}_{n-\ell+\varepsilon(\ell),\ell k}^k$ are projective in $\cO_k^{fin}$.
\end{theorem}
\begin{proof}
 Consider a diagram
 \begin{equation}\label{diag:P_proj}
  \xymatrix{
  & \widehat{P}_{n-\ell+\varepsilon(\ell),\ell k}^k \ar[d]^q \\
  A \ar[r]_{\pi} & B \\
  }
 \end{equation}
in $\cO_k^{fin}$ where $\pi$ is surjective and we may assume $q\neq 0$. We first consider the case $\ell=0$. Since there are no non-zero morphisms between the subcategories $(\cO_k^{fin})_e$ for different $e$'s, the diagram restricts to a surjection $\pi^0: A^0\twoheadrightarrow B^0$ and to $q:\widehat{P}_n^k\rightarrow B^0$ where $A^0=\ker_A E_0$ and $B^0=\ker_B E_0$. Then the diagram further restricts to a diagram
\begin{equation*}
  \xymatrix{
  &P_n \ar[d]^{q_{[0]}} \\
  A^0_{[0]} \ar[r]_{\pi_{[0]}} & B^0_{[0]} \\
  }
 \end{equation*}
 of $\gl(1|1)$-module homomorphisms, with $\pi_{[0]}$ surjective.

 Now since $N_0$ also acts semisimply on $A$ and $B$, $A^0_{[0]}$ and $B^0_{[0]}$ are objects of the category $\cO$. Since $P_n$ is projective in $\cO$, there is a $\gl(1|1)$-module map $\widetilde{q}_{[0]}: P_n\rightarrow A^0_{[0]}$ such that $q_{[0]}=\pi_{[0]}\circ\widetilde{q}_{[0]}$. Then Lemma \ref{lem:usefullemma_1} implies that positive modes from $\widehat{\gl(1|1)}$ annihilate $A^0_{[0]}$, so the universal property of generalized Verma $\widehat{\gl(1|1)}$-modules shows that $\widetilde{q}_{[0]}$ extends to a homomorphism
 \begin{equation*}
  \widetilde{q}: \widehat{P}_n^k\rightarrow A^0\hookrightarrow A.
 \end{equation*}
Since $P_n$ generates $\widehat{P}_n^k$, it follows that $q=\pi\circ\widetilde{q}$, and we have shown that $\widehat{P}_n^k$ is projective in $\cO_k^{fin}$.

 Now for $\ell\neq 0$, we apply $\sigma^{-\ell}$ or $w^{-1}\circ\sigma^{-\ell}$ to the diagram \eqref{diag:P_proj} to get
 \begin{equation*}
  \xymatrix{
  & \widehat{P}_{\pm n}^k \ar[d]^q \\
  \widetilde{A} \ar[r]_{\pi} & \widetilde{B} \\
  }
 \end{equation*}
 with $\pi$ still surjective. Since the spectral flows $\widetilde{A}$ and $\widetilde{B}$ are still objects of $\cO_k^{fin}$ and since $\widehat{P}_{\pm n}^k$ is projective in $\cO_k^{fin}$, we get $\widetilde{q}$ such that $q=\pi\circ\widetilde{q}$. Applying $\sigma^\ell$ or $\sigma^\ell\circ w$ then shows that $\widetilde{q}$ defines a homomorphism $\widehat{P}_{n-\ell+\varepsilon(\ell),\ell k}^k\rightarrow A$, so $\widehat{P}_{n-\ell+\varepsilon(\ell),\ell k}^k$ is projective in $\cO_k^{fin}$.
\end{proof}

Now we can prove that the projective modules $\widehat{P}_{n-\ell+\varepsilon(\ell),\ell k}^k$ are projective covers:
\begin{theorem}\label{projectivecovers}
 For $n\in\CC$ and $\ell\in\ZZ$, $\widehat{P}_{n-\ell+\varepsilon(\ell),\ell k}^k$ is a projective cover  in $\cO_k^{fin}$ of the atypical irreducible module $\widehat{A}_{n-\ell+\varepsilon(\ell),\ell k}^k$.
\end{theorem}
\begin{proof}
 The exact sequences \eqref{extatypical}, \eqref{exactsequencesimple}, \eqref{proj1}, and \eqref{proj4} show that there is a surjective map $\pi:\widehat{P}^k_{n-\ell+\varepsilon(\ell),\ell k}\rightarrow \widehat{A}^k_{n-\ell+\varepsilon(\ell),\ell k}$. Then if $q: P\rightarrow\widehat{A}^k_{n-\ell+\varepsilon(\ell),\ell k}$ is any surjective morphism in $\cO_k^{fin}$ with $P$ projective, there is a homomorphism $\widetilde{q}: P\rightarrow\widehat{P}^k_{n-\ell+\varepsilon(\ell),\ell k}$ such that the diagram
 \[\xymatrix{
                &         P \ar[d]^{q}   \ar@{-->}[ld]_{\widetilde{q}}  \\
  \widehat{P}^k_{n-\ell+\varepsilon(\ell),\ell k}  \ar[r]_{\pi} & \widehat{A}^k_{n-\ell+\varepsilon(\ell),\ell k}             }
\]
commutes. We need to show that $\widetilde{q}$ is surjective.

Applying $\sigma^{-\ell}$ or $w^{-1}\circ\sigma^{-\ell}$ to the diagram reduces the surjectivity of $\widetilde{q}$ to the $\ell=0$ case. But if $\ell=0$, the Loewy diagram of $\widehat{P}_n^k$ from Proposition \ref{prop:Pnk_Loewy} shows that $\ker\pi$ is the unique maximal proper submodule of $\widehat{P}_n^k$. Since $q\neq 0$, the image of $\widetilde{q}$ is not contained in $\ker\pi$, and we conclude $\widetilde{q}$ is surjective.
\end{proof}

We now turn to the typical irreducibles $\widehat{V}_{n,e}^k$ for $e/k \notin \ZZ$; we will show that they are projective and thus are their own projective covers in $\cO_k^{fin}$. Since they are irreducible and since every module in $\cO_k^{fin}$ has finite length, it is sufficient to show that when $e/k\notin\ZZ$, then
\begin{equation}\label{eqn:no_ext}
 \Ext^1_{\cO_k^{fin}}( \widehat{V}_{n,e}^k, W)=0
\end{equation}
for any irreducible $V_k(\gl(1|1))$-module $W$. From the direct sum decomposition $\cO_k^{fin}$, it is clear that an extension of $\widehat{V}_{n,e}^k$ by $W$ must split unless perhaps $W$ is an object of $(\cO_k^{fin})_e$; so we may assume $W=\widehat{V}_{n',e}^k$ for some $n'\in\CC$. Furthermore, the commutation relations for $N_0$ show that there is a direct sum decomposition
\begin{equation*}
 (\cO_k^{fin})_e=\bigoplus_{\overline{n}\in\CC/\ZZ} (\cO_k^{fin})_{\overline{n},e}
\end{equation*}
where $(\cO_k^{fin})_{\overline{n},e}$ is the full subcategory consisting of modules in $(\cO_k^{fin})_{e}$ on which $N_0$ acts by eigenvalues from the coset $\overline{n}$. Thus we may assume $W$ in \eqref{eqn:no_ext} is isomorphic to $V_{n',e}$ with $n'-n\in\ZZ$. We divide the proof of \eqref{eqn:no_ext} into the cases $n'\neq n$ and $n'=n$ in the two following lemmas:

%
%
% Let $v_{n,e}$ denote a highest weight vector of the module $\widehat{V}_{n,e}$, it has the conformal weight
% \[
% \Delta_{n,e} = n\frac{e}{k} + \frac{1}{2}\frac{e^2}{k^2}.
% \]
\begin{lemma}\label{typicaldifferent}
If $e/k\notin \ZZ$ and $(n,e) \neq (n',e')$, then $\Ext^1_{\cO_k^{fin}}(\widehat{V}_{n,e}^k, \widehat{V}_{n',e'}^k) = 0$.
\end{lemma}
\begin{proof}
We need to prove that any exact sequence
\begin{equation}\label{typicalexact}
0 \rightarrow \widehat{V}_{n',e'}^k \xr{q} A \xr{\pi} \widehat{V}_{n,e}^k \rightarrow 0
\end{equation}
splits when $A$ is a module in $\cO_k^{fin}$. As discussed above, we may assume $e'=e$, so that $n'\neq n$. Since $A$ has a direct sum decomposition $A=\bigoplus_{\overline{h}\in\CC/\ZZ} A_{\overline{h}}$ where $A_{\overline{h}}=\bigoplus_{h\in\overline{h}} A_{[h]}$, \eqref{typicalexact} must split unless perhaps lowest conformal weights satisfy $\Delta_{n,e}-\Delta_{n',e}\in\ZZ$. By \eqref{eqn:Delta_n,e},
\begin{equation*}
 \Delta_{n,e}-\Delta_{n',e}=\frac{e}{k}(n-n')\neq 0,
\end{equation*}
so we may assume either $\Delta_{n,e}-\Delta_{n',e}<0$ or $\Delta_{n,e}-\Delta_{n',e}>0$.

If $\Delta_{n,e}-\Delta_{n',e}<0$, then the lowest conformal weight space of $A$ is isomorphic to $V_{n,e}$ as a $\gl(1|1)$-module, and $\pi$ restricts to an isomorphism on lowest conformal weight spaces. Thus the universal property of induced $\widehat{\gl(1|1)}$-modules yields a homomorphism $\sigma:\widehat{V}^k_{n,e}\rightarrow A$ splitting the sequence.

If $\Delta_{n,e}-\Delta_{n',e}>0$, then $q$ restricts to a $\gl(1|1)$-module isomorphism from $V_{n',e'}$ onto the lowest conformal weight space of $A$, and this restriction must be even or odd. Since $V_{n',e'}$ generates $\widehat{V}_{n',e'}^k$, it follows that $q$ is parity-homogeneous and $\Im\,q=\ker\pi$ is $\ZZ_2$-graded. Then $\pi$ factors as
\begin{equation*}
 A\twoheadrightarrow A/\ker\pi\xrightarrow{\cong} \widehat{V}^k_{n,e}
\end{equation*}
where the first map is even and second is even or odd because $A/\ker\pi$ is either $\widehat{V}^k_{n,e}$ or its parity-reversed version. Thus $\pi$ is also parity-homogeneous, and we may apply the contragredient functor to get an exact sequence
\begin{equation*}
 0\rightarrow(\widehat{V}^k_{n,e})'\xr{\pi'} A'\xr{q'}(\widehat{V}^k_{n',e})'\rightarrow 0.
\end{equation*}
Since $e/k\notin\ZZ$, $(\widehat{V}^k_{n,e})'\cong\widehat{V}_{-n,-e}^k$ and $(\widehat{V}^k_{n',e})'\cong\widehat{V}^k_{-n',-e}$ by Proposition \ref{prop:contra}. Then the previous case implies that the contragredient sequence splits, that is, we have $\sigma': (\widehat{V}_{n',e}^k)'\rightarrow A'$ such that $q'\circ\sigma'=\id$. Taking contragredients again and applying the natural isomorphism $\delta$, we see that
\begin{equation*}
 \sigma := \delta_{\widehat{V}^k_{n',e}}^{-1}\circ\sigma''\circ\delta_A: A\rightarrow \widehat{V}^k_{n',e}
\end{equation*}
satisfies $\sigma\circ q=\id$. Thus $\mathrm{im}\,q$ is a direct summand of $A$ isomorphic to $\widehat{V}_{n',e}^k$, and there is an isomorphism $\widehat{V}_{n,e}^k\cong\ker\sigma$ splitting \eqref{typicalexact}.
\end{proof}

\begin{lemma}\label{typicalself}
If $e \neq 0$, then $\Ext_{\cO_k^{fin}}^1(\widehat{V}^k_{n,e}, \widehat{V}^k_{n,e}) = 0$
\end{lemma}
\begin{proof}
Assume that we have an exact sequence
\begin{equation}\label{typicalsplit0}
0 \rightarrow \widehat{V}^k_{n,e} \rightarrow A \rightarrow \widehat{V}^k_{n,e} \rightarrow 0.
\end{equation}
where $A$ is a module in $\cO_k^{fin}$. Then there is an exact sequence of $\gl(1|1)$-modules
\begin{equation}\label{typicalsplit}
0 \rightarrow V_{n,e} \rightarrow A_{[\Delta_{n,e}]} \rightarrow V_{n,e} \rightarrow 0,
\end{equation}
where $A_{[\Delta_{n,e}]}$ is the lowest conformal weight space of $A$. Since $A$ is a module in $\cO_k^{fin}$, $A_{[\Delta_{n,e}]}$ is a $\gl(1|1)$-module in the category $\cO$. Then since $V_{n,e}$ is projective in $\cO$ for $e\neq 0$, (\ref{typicalsplit}) splits. By the universal property of generalized Verma modules, the splitting homomorphism $V_{n,e}\rightarrow A_{[\Delta_{n,e}]}$ then extends to a unique $\widehat{\gl(1|1)}$-homomorphism $\widehat{V}_{n,e}^k \rightarrow A$ that splits \eqref{typicalsplit0}.
%
% $A_{[\Delta_{n,e}]} = V_{n,e} \oplus V_{n,e}$ as $\gl(1|1)$-modules. By the universal property, there is a unique $V_k(\gl(1|1))$-module homomorphism $\widehat{V}_{n,e} \rightarrow A$ that extends the $\gl(1|1)$-module homomorphism $V_{n,e} \rightarrow A_{[\Delta_{n,\ell k}]}$. It follows that the exact sequence (\ref{typicalsplit0}) splits.
\end{proof}

Combining Lemmas \ref{typicaldifferent} and \ref{typicalself} with the discussion preceding Lemma \ref{typicaldifferent}, we conclude:
\begin{theorem}
For $n \in \CC$ and $e/k \notin \ZZ$, the typical irreducible module $\widehat{V}_{n,e}^k$ is projective in $\cO_k^{fin}$.
\end{theorem}

\begin{cor}
 Every irreducible $V_k(\gl(1|1))$-module has a projective cover in $\cO_k^{fin}$.
\end{cor}

\section{Fusion rules}

In this section, we compute tensor products of irreducible $\widehat{\gl(1|1)}$-modules. First we discuss some general results on intertwining operators and fusion rules for affine Lie superalgebras, and then we proceed to results in the  affine $\gl(1|1)$ case.

\subsection{Determining fusion rules for affine Lie superalgebras}\label{subsec:gen_fus_rules}

Here we collect some general results on determining fusion rules for affine vertex operator superalgebras. For vertex operator algebras, the general theory is developed in \cite{FZ1, Li} for the non-logarithmic case and \cite{HY} for the logarithmic case.

Let $\frg$ be a finite-dimensional Lie superalgebra with corresponding affine Lie superalgebra $\widehat{\frg}$, and let $V^k(\frg)$ be the level-$k$ affine vertex operator superalgebra. A $V^k(\frg)$-module $W$ is $\NN$-\textit{gradable} if it has an $\NN$-grading $W=\bigoplus_{i\in\NN} W(i)$ compatible with the $\ZZ_2$-grading such that
\begin{equation*}
 a_r\cdot W(i)\subseteq W(i-r)
\end{equation*}
for $a\in\frg$, $r\in\ZZ$, and $i\in\NN$, where $W(i)=0$ for $i<0$. Note that an $\NN$-grading satisfying these properties need not be unique.

Now suppose $\cY$ is an even or odd (possibly logarithmic) $V^k(\frg)$-module  intertwining operator of type $\binom{W_3}{W_1\,W_2}$. The intertwining operator Jacobi identity (see \cite[Def. 3.7]{CKM} for the proper sign factors in the superalgebra generality) implies the following commutator and iterate formulas:
\begin{equation}\label{eqn:intw_op_comm}
 (-1)^{\vert\cY\vert\vert a\vert} a_r\cY(w_1,x) -(-1)^{\vert a\vert\vert w_1\vert}\cY(w_1,x)a_r =\sum_{i\geq 0} \binom{r}{i} x^{r-i}\cY(a_i w_1,x)
\end{equation}
and
\begin{equation}\label{eqn:intwo_op_it}
 \cY(a_r w_1,x) =(-1)^{\vert\cY\vert\vert a\vert}\sum_{i\geq 0}\binom{r}{i} (-x)^i a_{r-i}\cY(w_1,x) -(-1)^{\vert a\vert\vert w_1\vert}\sum_{i\geq 0}\binom{r}{i}(-x)^{r-i}\cY(w_1,x)a_i
\end{equation}
for homogeneous $a\in\frg$, $w_1\in W_1$, and $r\in\ZZ$.

Suppose $W_1$ and $W_2$ are generalized Verma modules $\widehat{M}_1^k$ and $\widehat{M}_2^k$, respectively, and that $W_3$ is $\NN$-gradable such that each $W_3(i)$ is the sum of finitely many generalized $L_0$-eigenspaces. This means that the substitution $x\mapsto 1$ in $\cY$ is well defined (using the branch of logarithm $\log 1=0)$, and $\cY(w_1,1)w_2$ is a vector in the algebraic completion $\prod_{i\in\NN} W_3(i)$. Then the $r=0$ case of the commutator formula \eqref{eqn:intw_op_comm} implies that $\cY$ induces a $\frg$-module homomorphism
\begin{equation*}
 \pi(\cY): M_1\otimes M_2\rightarrow W_3(0)
\end{equation*}
with parity $\vert\cY\vert$ defined by
\begin{equation*}
 \pi(\cY)(m_1\otimes m_2) = \pi_0(\cY(m_1,1)m_2)
\end{equation*}
for $m_1\in M_1$, $m_2\in M_2$, where $\pi_0$ denotes projection onto $W_3(0)$ with respect to the $\NN$-grading of $W_3$. The following is essentially a special case of \cite[Prop. 24]{TW}, which is attributed to Nahm \cite{Na}:
\begin{prop}\label{prop:piY_surjective}
 If $\cY$ is a surjective intertwining operator, then $\pi(\cY)$ is a surjective $\frg$-module homomorphism.
\end{prop}
\begin{proof}
 Because $\cY$ is surjective, $W_3(0)$ is spanned by vectors of the form $\pi_0(\cY(w_1,1)w_2)$ for $w_1\in\widehat{M}_1^k$ and $w_2\in\widehat{M}_2^k$. Thus we need to show that
 \begin{equation*}
  \pi_0(\cY(w_1,1)w_2)\in\Span\lbrace \pi_0(\cY(m_1,1)m_2)\,\vert\,m_1\in M_1, m_2\in M_2\rbrace
 \end{equation*}
for all $w_1\in\widehat{M}_1^k$, $w_2\in\widehat{M}_2^k$. This is true by definition when $w_1\in M_1$ and $w_2\in M_2$.

Now assume that for some $w_2\in\widehat{M}_2^k$, we already know that $\pi_0(\cY(m_1,1)w_2)\in\Im\,\pi(\cY)$ for all $m_1\in M_1$. Then \eqref{eqn:intw_op_comm} implies that
\begin{align*}
 (-1)^{\vert a\vert\vert m_1\vert}\pi_0(\cY(m_1,1)a_{-r}w_2) & =(-1)^{\vert\cY\vert\vert a\vert} \pi_0(a_{-r}\cY(m_1,1)w_2)-\sum_{i\in\NN}\binom{-r}{i}\pi_0(\cY(a_i m_1,1)w_2)\nonumber\\
 & = -\pi_0(\cY(a_0 m_1,1)w_2)\in\Im\,\pi(\cY)
\end{align*}
for homogeneous $a\in\frg$, $m_1\in M_1$, and $r\in\ZZ_+$. Since $\widehat{M}_2^k$ is generated by $M_2$ under the action of the modes $a_{-r}$, this shows that $\pi_0(\cY(m_1,1)w_2)\in\Im\,\pi(\cY)$ for all $m_1\in M_1$ and all $w_2\in \widehat{M}_2^k$.

Now assume that for some homogeneous $w_1\in\widehat{M}_1^k$, we know  $\pi_0(\cY(w_1,1)w_2)\in\Im\,\pi(\cY)$ for all $w_2\in \widehat{M}_2^k$. Then \eqref{eqn:intwo_op_it} implies that
\begin{align*}
\pi_0( \cY( & a_{-r} w_1,1)w_2)\nonumber\\ &=(-1)^{\vert\cY\vert\vert a\vert}\sum_{i\geq 0}\binom{-r}{i} \pi_0(a_{-r-i}\cY(w_1,1)w_2)-(-1)^{\vert a\vert\vert w_1\vert}\sum_{i\geq 0}\binom{-r}{i}\pi_0(\cY(w_1,1)a_i w_2)\nonumber\\
& = -(-1)^{\vert a\vert\vert w_1\vert}\sum_{i\geq 0}\binom{-r}{i}\pi_0(\cY(w_1,1)a_i w_2)\in\Im\,\pi(\cY)
\end{align*}
for homogeneous $a\in\frg$ $r\in\ZZ_+$, and $w_2\in W_2$. Since $\widehat{M}_1^k$ is a generalized Verma module, this shows that $\pi_0(\cY(w_1,1)w_2)\in\Im\,\pi(\cY)$ for all $w_1\in \widehat{M}_1^k$, $w_2\in \widehat{M}^k_2$, as required.
\end{proof}

For $V^k(\frg)$-modules $\widehat{M}_1^k$, $\widehat{M}_2^k$ and $W_3$ as above, it turns out that the even linear map $\pi$ from intertwining operators of type $\binom{W_3}{\widehat{M}_1^k\,\widehat{M}_2^k}$ to $\Hom_\frg(M_1\otimes M_2, W_3(0))$ is an isomorphism if $W_3$ is the contragredient $(\widehat{M}_3^k)'$ of a generalized Verma module, in which case $W_3(0)=M_3^*$.
% Recall that the contragredient of a generalized $V^k(\frg)$-module $W=\bigoplus_{h\in\CC} W_{[h]}$ is the graded dual $W'=\bigoplus_{h\in\CC} W_{[h]}^*$ equipped with the vertex operator
% \begin{equation}\label{eqn:contra_vrtx_op}
%  \langle Y_{W'}(v,x)w', w\rangle =(-1)^{\vert v\vert\vert w'\vert}\langle w', Y_W(e^{x L_1}(-x^{-2})^{L_0} v,x^{-1}w\rangle
% \end{equation}
% for homogeneous $v\in V^k(\frg)$, $w'\in W'$. (Here we use the convention of \cite{CKM}. Note that there are several different yet equivalent definitions in the literature for the contragredient of a module for a vertex operator superalgebra; see \cite[Rem. 3.5]{CKM} for more details.)
In fact, the following theorem is the affine Lie superalgebra case of (the superalgebra generalization of) \cite[Thm. 6.6]{HY}, which generalizes \cite[Thm. 2.11]{Li} and \cite[Lem. 2.19]{FZ2}  to logarithmic intertwining operators:
\begin{theorem}\label{thm:gen_fus_rules}
 Suppose $M_1$, $M_2$, and $M_3$ are finite-dimensional $\frg$-modules and $f: M_1\otimes M_2\rightarrow M_3^*$ is a $\frg$-module homomorphism. Then there is a unique intertwining operator $\cY$ of type $\binom{(\widehat{M}_3^k)'}{\widehat{M}_1^k\,\widehat{M}_2^k}$ such that $\pi(\cY)=f$.
\end{theorem}

To make the proof of \cite[Thm. 6.6]{HY} in the affine Lie superalgebra special case more concrete, we sketch a construction of $\cY$ from a homogeneous $f$. We first use \eqref{eqn:intw_op_comm} and \eqref{eqn:aff_contra} to define a sequence of linear maps
$$f_i: M_1\otimes M_2\rightarrow(\widehat{M}_3^k)'(i)=\widehat{M}_3^k(i)^*$$
recursively. We start with $f_0=f$, and then assuming $f_0,\ldots, f_{i-1}$ have been defined, we set
\begin{equation*}
 \langle f_i(m_1\otimes m_2), a_{-r} w_3\rangle =-(-1)^{\vert a\vert(\vert m_1\vert+\vert m_2\vert)}\langle f_{i-r}(a_0 m_1\otimes m_2),w_3\rangle
\end{equation*}
for homogeneous $m_1\in M_1$, $m_2\in M_2$, $a\in \frg$, $1\leq r\leq i$, and $w_3\in\widehat{M}_3^k(i-r)$. Once it is shown that the maps $f_i$ are well defined, we can then set
\begin{equation*}
 \cY(m_1,x)m_2 =\sum_{i\in\NN} x^{L_0}f_i(x^{-L_0} m_1\otimes x^{-L_0} m_2): M_1\otimes M_2\rightarrow(\widehat{M}_3^k)'[\log x]\lbrace x\rbrace.
\end{equation*}
Next, we use the method of \cite{MY} to extend $\cY$ first to $M_1\otimes\widehat{M}_2^k$, and then to $\widehat{M}_1^k\otimes\widehat{M}_2^k$. Both extensions are defined recursively: assuming we have already defined $\cY(m_1,x)w_2$ for $m_1\in M_1$, $w_2\in\widehat{M}_2^k$, we use the commutator formula \eqref{eqn:intw_op_comm} to define
\begin{align*}
 \cY(m_1,x)a_{-r}w_2 =(-1)^{\vert a\vert\vert m_1\vert}\left((-1)^{\vert\cY\vert\vert a\vert} a_{-r}\cY(w_1,x)w_2- x^{-r}\cY(a_0 m_1,x)w_2\right)
\end{align*}
for homogeneous $m_1\in M_1$, $a\in\frg$, and $r\in\ZZ_+$. Finally, assuming we have defined $\cY(w_1,x)w_2$ for some homogeneous $w_1\in\widehat{M}_1^k$ and all $w_2\in \widehat{M}^k_2$, we use \eqref{eqn:intwo_op_it} to define
\begin{align*}
 \cY(a_{-r} w_1,x)
 & =(-1)^{\vert f\vert \vert a\vert}\sum_{i\geq 0} \binom{-r}{i}(-x)^i a_{-r-i}\cY(w_1,x)\nonumber\\
 &\qquad-(-1)^{\vert a\vert\vert w_1\vert}\sum_{i\geq 0}\binom{-r}{i} (-x)^{-r-i}\cY(w_1,x)a_i
\end{align*}
for homogeneous $a\in\frg$ and $r\in\ZZ_+$. Once it is shown that these extensions are well defined, one can prove that $\cY$ is an intertwining operator, as in \cite[Thm. 6.2]{MY}. Note that $\cY$ has the same parity of $\vert f\vert$; the uniqueness of $\cY$ follows because the construction of $\cY$ is forced by the formulas \eqref{eqn:intw_op_comm}, \eqref{eqn:intwo_op_it}, and \eqref{eqn:contra_vrtx_op}.

\subsection{Fusion rules in $KL_k$ and $\cO_k^{fin}$}

In this section, we compute all tensor products of irreducible $V_k(\gl(1|1))$-modules. For fusion rules involving the irreducible modules $\widehat{A}^k_{0,\pm k}$, we will need the following lemma on the kernels of $\gl(1|1)$-module maps $\pi(\cY)$, where $\cY$ is an intertwining operator involving $\widehat{A}^k_{0,\pm k}$. In the statement, we use $v_{n,e}$ to denote a highest-weight vector in the $\gl(1|1)$-module $V_{n,e}$:
\begin{lemma}\label{lem:sing_vector} \hspace{2em}
 \begin{enumerate}
  \item Suppose $\cY$ is an even intertwining operator of type $\binom{W_3}{\widehat{A}^k_{0,k}\,W_2}$ where $W_2$ and $W_3$ are $\NN$-gradable $V_k(\gl(1|1))$-modules with $W_2(0)=V_{n',e'}$ for $e'\neq 0$. Then $v_{0,k}\otimes v_{n',e'}\in\ker\pi(\cY)$.

  \item Suppose $\cY$ is an even intertwining operator of type $\binom{W_3}{W_1\,\widehat{A}^k_{0,-k}}$ where $W_1$ and $W_3$ are $\NN$-gradable $V_k(\gl(1|1))$-modules with $W_1(0)=V_{n,e}$ for $e\neq 0$. Then $\psi^- v_{n,e}\otimes\psi^- v_{0,-k}\in\ker\pi(\cY)$.
 \end{enumerate}
\end{lemma}
\begin{proof}
 We use explicit singular vectors in $\widehat{V}^k_{0,\pm k}$: it is straightforward to show that
\begin{equation*}
 \psi_{-1}^+ v_{0,k}\in\widehat{V}^k_{0,k}\qquad\text{and}\qquad (k\psi^-_{-1}+E_{-1}\psi^-_0) v_{0,-k}\in\widehat{V}^k_{0,-k}
\end{equation*}
vanish in the irreducible quotients $\widehat{A}^k_{0,\pm k}$. Thus in the first case we use \eqref{eqn:intwo_op_it} to calculate
\begin{align*}
 0 & = \pi_0\left(\cY(\psi_{-1}^+ v_{0,k},1)\psi_0^- v_{n',e'}\right)\nonumber\\
 &=\pi_0\sum_{i\geq 0} \left(\psi^+_{-1-i}\cY(v_{0,k},1)\psi^-_0 v_{n',e'} +(-1)^{\vert\psi^+\vert\vert v_{0,k}\vert}\cY(v_{0,k},1)\psi^+_i \psi^-_0 v_{n',e'}\right)\nonumber\\
& = (-1)^{\vert v_{0,k}\vert}\pi_0\left(\cY(v_{0,k},1)E_0 v_{n',e'}\right) =e'(-1)^{\vert v_{0,k}\vert}\pi(\cY)(v_{0,k}\otimes v_{n',e'}).
\end{align*}
Since $e'\neq 0$, this shows that $v_{0,k}\otimes v_{n',e'}\in\ker\pi(\cY)$. For the second case, we use \eqref{eqn:intw_op_comm} to calculate
\begin{align*}
 0 & =\pi_0\left(\cY(\psi^-_0 v_{n,e},1)(k\psi^-_{-1}+E_{-1}\psi^-_0) v_{0,-k}\right)\nonumber\\
 & =k(-1)^{\vert\psi^-\vert(\vert\psi^-\vert+\vert v_{n,e}\vert)}\pi_0\bigg(\psi_{-1}^-\cY(\psi_0^- v_{n,e},1)v_{0,-k} -\sum_{i\geq 0} (-1)^i\cY(\psi^-_i\psi^-_0 v_{n,e},1)v_{0,-k}\bigg)\nonumber\\
 & \qquad +(-1)^{\vert E\vert(\vert\psi^-\vert+\vert v_{n,e}\vert)}\pi_0\bigg( E_{-1}\cY(\psi_0^- v_{n,e},1)\psi_0^- v_{0,-k}-\sum_{i\geq 0} (-1)^i\cY(E_i\psi_0^- v_{n,e},1)\psi^-_0 v_{0,-k}\bigg)\nonumber\\
 & =-e\,\pi(\cY)(\psi^- v_{n,e}\otimes\psi^- v_{0,-k}).
\end{align*}
Since $e\neq 0$, $\psi^- v_{n,e}\otimes\psi^- v_{0,-k}\in\ker\pi(\cY)$.
\end{proof}

First we compute the tensor products of atypical irreducible $V_k(\gl(1|1))$-modules:
\begin{theorem}\label{thm:atyp_atyp_fusion}
 The atypical irreducible $V_k(\gl(1|1))$-modules are simple currents with fusion rules
 \begin{equation*}
  \widehat{A}_{n,\ell k}^k\boxtimes\widehat{A}_{n',\ell' k}^k \cong \widehat{A}_{n+n'-\varepsilon(\ell,\ell'),(\ell+\ell')k}^k
 \end{equation*}
for $n,n'\in\CC$, $\ell,\ell'\in\ZZ$, where $\varepsilon(\ell,\ell')=\varepsilon(\ell)+\varepsilon(\ell')-\varepsilon(\ell+\ell')$.
\end{theorem}
\begin{proof}
 We will prove two special cases of the fusion rules: the $\ell=0$ case
 \begin{equation}\label{eqn:ell=0_case}
  \widehat{A}_{n,0}^k\boxtimes\widehat{A}_{n',\ell'k}^k\cong\widehat{A}^k_{n+n',\ell'k},
 \end{equation}
and also
\begin{equation}\label{eqn:ell_case}
 \widehat{A}_{-\frac{\ell}{2}+\varepsilon(\ell),\ell k}^k\boxtimes\widehat{A}_{-\frac{\ell'}{2}+\varepsilon(\ell'),\ell' k}^k\cong \widehat{A}_{-\frac{\ell+\ell'}{2}+\varepsilon(\ell+\ell'),(\ell+\ell')k}^k
\end{equation}
for $\ell,\ell'\in\ZZ$. The general formula then follows from these by associativity and commutativity of tensor products:
\begin{align*}
 \widehat{A}_{n,\ell k}^k\boxtimes\widehat{A}_{n',\ell' k}^k & \cong (\widehat{A}^k_{n+\frac{\ell}{2}-\varepsilon(\ell),0}\boxtimes\widehat{A}_{-\frac{\ell}{2}+\varepsilon(\ell),\ell k}^k)\boxtimes(\widehat{A}_{n'+\frac{\ell'}{2}-\varepsilon(\ell'),0}^k\boxtimes\widehat{A}_{-\frac{\ell'}{2}+\varepsilon(\ell'),\ell' k}^k)\nonumber\\
 & \cong(\widehat{A}_{n+\frac{\ell}{2}-\varepsilon(\ell),0}^k\boxtimes\widehat{A}_{n'+\frac{\ell'}{2}-\varepsilon(\ell'),0}^k)\boxtimes(\widehat{A}_{-\frac{\ell}{2}+\varepsilon(\ell),\ell k}^k\boxtimes\widehat{A}_{-\frac{\ell'}{2}+\varepsilon(\ell'),\ell' k}^k)\nonumber\\
 & \cong \widehat{A}^k_{n+n'+\frac{\ell+\ell'}{2}-\varepsilon(\ell)-\varepsilon(\ell'),0}\boxtimes\widehat{A}^k_{-\frac{\ell+\ell'}{2}+\varepsilon(\ell+\ell'),(\ell+\ell')k}\nonumber\\
 & \cong\widehat{A}^k_{n+n'-\varepsilon(\ell+\ell'),(\ell+\ell')k}
\end{align*}
for $n,n'\in\CC$ and $\ell,\ell\in\ZZ$.

To prove \eqref{eqn:ell=0_case}, we first take $\ell'=0$. Then Proposition \ref{prop:piY_surjective} applied to the surjective tensor product intertwining operator yields a (non-zero) surjective $\gl(1|1)$-module homomorphism
\begin{equation*}
 A_n\otimes A_{n'}\cong A_{n+n'}\twoheadrightarrow(\widehat{A}^k_{n,0}\boxtimes\widehat{A}^k_{n',0})(0)=(\widehat{A}^k_{n,0}\boxtimes\widehat{A}^k_{n',0})_{[0]}
\end{equation*}
(recall also Lemma \ref{lem:usefullemma_1} and Remark \ref{rem:useful}). The universal property of induced $\widehat{\gl(1|1)}$-modules then leads to a non-zero homomorphism $\widehat{A}^k_{n+n',0}\rightarrow\widehat{A}^k_{n,0}\boxtimes\widehat{A}^k_{n',0}$ which is injective because $\widehat{A}^k_{n+n',0}$ is simple. It is also surjective because Lemma \ref{lem:usefullemma_1} says that $\widehat{A}^k_{n,0}\boxtimes\widehat{A}^k_{n',0}$ is generated by its weight-$0$ subspace, so we have proved the $\ell'=0$ case of \eqref{eqn:ell=0_case}. This also shows that $\widehat{A}^k_{n,0}$ is a simple current for $n\in\CC$, with tensor inverse $\widehat{A}^k_{-n,0}$.

For $\ell'\neq 0$, we now know that $\widehat{A}^k_{n,0}\boxtimes\widehat{A}^k_{n',\ell' k}$ is simple because $\widehat{A}^k_{n,0}$ is a simple current. Since Proposition \ref{prop:piY_surjective} gives a surjective $\gl(1|1)$-module homomorphism
\begin{equation*}
 A_n\otimes V_{n',\ell'k}\cong V_{n+n',\ell'k}\twoheadrightarrow(\widehat{A}^k_{n,0}\boxtimes\widehat{A}^k_{n',\ell'k})(0),
\end{equation*}
it follows that $\widehat{A}^k_{n,0}\boxtimes\widehat{A}^k_{n',\ell'k}$ is the unique irreducible quotient $\widehat{A}^k_{n+n',\ell'k}$ of $\widehat{V}^k_{n+n',\ell'k}$. This finishes the proof of  \eqref{eqn:ell=0_case}.

For \eqref{eqn:ell_case}, we first take $\ell=1$ and $\ell'=-1$. In this case the tensor product intertwining operator $\cY$ of type $\binom{\widehat{A}^k_{0,k}\boxtimes\widehat{A}^k_{0,-k}}{\widehat{A}^k_{0,k}\,\widehat{A}^k_{0,-k}}$ induces a surjective $\gl(1|1)$-homomorphism
\begin{equation*}
 \pi(\cY): V_{0,k}\otimes V_{0,-k}\cong P_0\twoheadrightarrow(\widehat{A}^k_{0,k}\boxtimes\widehat{A}^k_{0,-k})(0)=(\widehat{A}^k_{0,k}\boxtimes\widehat{A}^k_{0,-k})_{[0]}
\end{equation*}
(again recall Lemma \ref{lem:usefullemma_1} and Remark \ref{rem:useful}). The universal property of induced modules then yields a homomorphism $\Pi: \widehat{P}_0^k\rightarrow\widehat{A}^k_{0,k}\boxtimes\widehat{A}^k_{0,-k}$ which is surjective by Lemma \ref{lem:usefullemma_1}. It will induce the required isomorphism $\widehat{A}^k_{0,0}\cong\widehat{A}^k_{0,k}\boxtimes\widehat{A}^k_{0,-k}$ if $\ker\Pi$ is the (unique) maximal proper submodule of $\widehat{P}_0^k$. To prove this, we just need to show that $\ker \pi(\cY)$ contains the unique maximal proper submodule of $P_0$. In fact, Lemma \ref{lem:sing_vector} says that
\begin{equation*}
 v_{0,k}\otimes v_{0,-k},\,\psi^-v_{0,k}\otimes\psi^-v_{0,-k}\in\ker\pi(\cY),
\end{equation*}
and it is easy to see that these two vectors generate the maximal proper submodule of $V_{0,k}\otimes V_{0,-k}\cong P_0$. This proves the $\ell=1$, $\ell'=-1$ case of \eqref{eqn:ell_case}, and we have also now shown that $\widehat{A}^k_{0,\pm k}$ are mutually inverse simple currents.

Now since $\widehat{A}^k_{0,k}$ generates a group of simple currents, \eqref{eqn:ell_case} for general $\ell,\ell'\in\ZZ$ will follow if we can show that
\begin{equation*}
 (\widehat{A}^k_{0,k})^{\boxtimes\ell} \cong \widehat{A}^k_{-\frac{\ell}{2}+\varepsilon(\ell),\ell k}
\end{equation*}
for $\ell\in\ZZ$. Since this relationship holds for $\ell=-1,0,1$, it will hold in general by induction on $\ell$ if we can show that
\begin{equation}\label{eqn:ind_con}
 \widehat{A}^k_{0,k}\boxtimes\widehat{A}^k_{-\frac{\ell}{2}+\frac{1}{2},\ell k} \cong  \widehat{A}^k_{-\frac{\ell}{2},(\ell+1) k}\quad\text{and}\quad\widehat{A}^k_{\frac{\ell}{2}-\frac{1}{2},-\ell k}\boxtimes\widehat{A}^k_{0,-k}\cong\widehat{A}^k_{\frac{\ell}{2},-(\ell+1)k}
\end{equation}
for $\ell\in\ZZ_+$.

To prove \eqref{eqn:ind_con}, we first note that both tensor product modules are simple because $\widehat{A}^k_{0,\pm k}$ are simple currents. Then Proposition \ref{prop:piY_surjective} again shows that the minimal conformal weight spaces of the tensor product modules are $\gl(1|1)$-module quotients of
\begin{equation*}
 V_{0,k}\otimes V_{-\frac{\ell}{2}+\frac{1}{2},\ell k}\cong V_{-\frac{\ell}{2}+1,(\ell+1)k}\oplus V_{-\frac{\ell}{2},(\ell+1) k}
 \end{equation*}
 and
 \begin{equation*}
 V_{\frac{\ell}{2}-\frac{1}{2},-\ell k}\otimes V_{0,-k}\cong V_{\frac{\ell}{2},-(\ell+1)k}\oplus V_{\frac{\ell}{2}-1,-(\ell+1)k},
 \end{equation*}
respectively. Thus we just need to show that the kernels of the respective $\gl(1|1)$-surjections $\pi(\cY)$ agree with $V_{-\frac{\ell}{2}+1,(\ell+1)k}$ and $V_{\frac{\ell}{2}-1,-(\ell+1)k}$. For the first case, Lemma \ref{lem:sing_vector}(1) shows that $v_{0,k}\otimes v_{-\frac{\ell}{2}+\frac{1}{2},\ell k}\in\ker\pi(\cY)$; this is a highest-weight vector generating $V_{-\frac{\ell}{2}+1,(\ell+1)k}$, so the first case of \eqref{eqn:ind_con} is proved. For the second case, Lemma \ref{lem:sing_vector}(2) shows that $\psi_0^- v_{\frac{\ell}{2}-\frac{1}{2},-\ell k}\otimes\psi_0^- v_{0,-k}\in\ker\pi(\cY)$. As this is a lowest-weight vector generating $V_{\frac{\ell}{2}-1,-(\ell+1)k}$, the second case of \eqref{eqn:ind_con} is proved. This completes the proof of the theorem.
%
%
%
%
% the calculation \eqref{eqn:1st_sing_calc}, with $v_-$ now a highest-weight vector generating $\widehat{A}^k_{-\frac{\ell}{2}+\frac{1}{2},\ell k}$, shows that $v_+\otimes v_-\in\ker\pi(\cY)$. Since $v_+\otimes v_-$ is a  For the second case, the calculation \eqref{eqn:2nd_sing_calc}, with $v_+$ now a highest-weight vector generating $\widehat{A}^l_{\frac{\ell}{2}-\frac{1}{2},-\ell k}$, shows that
\end{proof}

Next, we compute the tensor products of atypical with typical irreducible $V_k(\gl(1|1))$-modules:
\begin{theorem}\label{thm:atyp_typ_fusion}
 For $n,n'\in\CC$, $\ell\in\ZZ$, and $e'/k\notin\ZZ$,
 \begin{equation*}
  \widehat{A}^k_{n,\ell k}\boxtimes\widehat{V}^k_{n',e'}\cong\widehat{V}^k_{n+n'-\varepsilon(\ell), e'+\ell k}.
 \end{equation*}
\end{theorem}
\begin{proof}
 We will prove the special cases
 \begin{equation}\label{eqn:ell=0_case_2}
  \widehat{A}^k_{n,0}\boxtimes\widehat{V}^k_{n',e'}\cong\widehat{V}^k_{n+n',e'}
 \end{equation}
and
\begin{equation}\label{eqn:ell_case_2}
 \widehat{A}^k_{0,\pm k}\boxtimes\widehat{V}^k_{n',e'}\cong\widehat{V}^k_{n'\mp\frac{1}{2},e'\pm k}.
\end{equation}
The general case then follows from these fusion rules together with associativity of the tensor product and the fusion rules for atypical modules from Theorem \ref{thm:atyp_atyp_fusion}:
\begin{align*}
 \widehat{A}^k_{n,\ell k}\boxtimes \widehat{V}^k_{n',e'} & \cong (\widehat{A}^k_{-\frac{\ell}{2}+\varepsilon(\ell),\ell k}\boxtimes \widehat{A}^k_{n+\frac{\ell}{2}-\varepsilon(\ell), 0})\boxtimes \widehat{V}^k_{n',e'}\nonumber\\
 & \cong(\widehat{A}^k_{0,\pm k})^{\boxtimes\vert\ell\vert}\boxtimes \widehat{V}^k_{n+n'+\frac{\ell}{2}-\varepsilon(\ell),e'}\nonumber\\
 &\cong \widehat{V}^k_{n+n'+\frac{\ell}{2}-\varepsilon(\ell)\mp\frac{\vert\ell\vert}{2}, e'\pm\vert\ell\vert k} = \widehat{V}^k_{n+n'-\varepsilon(\ell),e'+\ell k}
\end{align*}
for all $n,n'\in\CC$, $\ell\in\ZZ$, and $e'/k\notin\ZZ$.

For \eqref{eqn:ell=0_case_2}, Proposition \ref{prop:piY_surjective} yields a surjective $\gl(1|1)$-homomorphism
\begin{equation*}
 A_n\otimes V_{n',e'}\cong V_{n+n',e'}\twoheadrightarrow (\widehat{A}^k_{n,0}\boxtimes\widehat{V}^k_{n',e'})(0),
\end{equation*}
which lifts to a non-zero map $\Pi: \widehat{V}^k_{n+n',e'}\rightarrow\widehat{A}^k_{n,0}\boxtimes\widehat{V}^k_{n',e'}$ by the universal property of induced $\widehat{\gl(1|1)}$-modules. Since $e'/k\notin\ZZ$ and $\widehat{A}^k_{n,0}$ is a simple current, both domain and codomain of $\Pi$ are simple; therefore $\Pi$ is an isomorphism.

For \eqref{eqn:ell_case_2}, Proposition \ref{prop:piY_surjective} yields surjective $\gl(1|1)$-homomorphisms
\begin{align*}
\pi(\cY_+)  :   V_{0, k}\otimes V_{n',e'} & \cong V_{n'+\frac{1}{2},e'+ k}\oplus V_{n'-\frac{1}{2},e'+ k}\twoheadrightarrow(\widehat{A}^k_{0, k}\boxtimes\widehat{V}^k_{n',e'})(0)\nonumber\\
\pi(\cY_-)  :  V_{n',e'}\otimes V_{0,-k} & \cong V_{n'+\frac{1}{2},e'-k}\oplus V_{n'-\frac{1}{2},e'-k}\twoheadrightarrow(\widehat{V}_{n',e'}^k\boxtimes\widehat{A}^k_{0,-k})(0).
\end{align*}
The modules $\widehat{A}^k_{0, \pm k}\boxtimes\widehat{V}^k_{n',e'}$ are again irreducible, so it is enough to determine the kernels of $\pi(\cY_{\pm})$. Since $e'\neq 0$, Lemma \ref{lem:sing_vector}(1) shows that $v_{0,k}\otimes v_{n',e'}\in\ker\pi(\cY_+)$, and thus $\ker\pi(\cY_+)=V_{n'+\frac{1}{2},e'+k}$, while Lemma \ref{lem:sing_vector}(2) shows that $\psi^- v_{n',e'}\otimes\psi^- v_{0,-k}\in\ker\pi(\cY_-)$, and thus $\ker\pi(\cY_-)=V_{n'-\frac{1}{2},e'-k}$. This completes the proof of \eqref{eqn:ell_case_2} and of the theorem.
%
%
%
% The singular vector calculation \eqref{eqn:1st_sing_calc} (with $v_-$ now a highest-weight vector of $V_{n',e'}$) shows that $v_+\otimes v_-\in\ker\pi(\cY_+)$, and thus $\ker\pi(\cY_+)=V_{n'+\frac{1}{2},e'+k}$. On the other hand, \eqref{eqn:2nd_sing_calc} (with $v_+$ now a highest-weight vector in $V_{n',e'}$) shows that
\end{proof}

Finally, we compute products of typical irreducible $V_k(\gl(1|1))$-modules, although the following theorem also incorporates some fusion rules involving reducible Verma modules:
\begin{theorem}\label{thm:typ_typ_fusion}
For $n,n'\in\CC$,
 \begin{equation*}
  \widehat{V}^k_{n,e}\boxtimes\widehat{V}^k_{n',e'}\cong\left\lbrace\begin{array}{lll}
  \widehat{V}_{n+n'+\frac{1}{2},e+e'}^k\oplus\widehat{V}_{n+n'-\frac{1}{2},e+e'}^k & \text{if} & (e+e')/k\notin\ZZ \\
  \widehat{P}^k_{n+n'} & \text{if} & e+e'=0 \\
         \widehat{P}^k_{ n+n'+\varepsilon((e+e')/k),e+e'} & \text{if} & (e+e')/k\in\ZZ\setminus\lbrace 0\rbrace  \,\,\text{and}\,\,
                 e'/k\notin\ZZ\\                                                   \end{array}
\right. .
 \end{equation*}
\end{theorem}
\begin{proof}
 It is straightforward to compute that, as $\gl(1|1)$-modules,
 \begin{equation*}
  V_{n,e}\otimes V_{n',e'}\cong\left\lbrace\begin{array}{lll}
                                            V_{n+n'+\frac{1}{2},e+e'}\oplus V_{n+n'-\frac{1}{2},e+e'} & \text{if} & e+e'\neq 0\\
                                            P_{n+n'} & \text{if} & e+e'=0\\
                                           \end{array}
\right. .
 \end{equation*}
In particular, $V_{n,e}\otimes V_{n',e'}\cong\widehat{P}^k(0)$, where $\widehat{P}^k$ is the generalized Verma module
\begin{equation*}
 \widehat{P}^k=\left\lbrace\begin{array}{lll}
  \widehat{V}_{n+n'+\frac{1}{2},e+e'}^k\oplus\widehat{V}_{n+n'-\frac{1}{2},e+e'}^k & \text{if} & e+e'\neq 0\\
         \widehat{P}^k_{ n+n'} & \text{if} & e+e'=0                                            \end{array}
\right. .
\end{equation*}
When $(e+e')/k\notin\ZZ\setminus\lbrace 0\rbrace$, therefore, $\widehat{P}^k$ is projective in $\cO_k^{fin}$ (by the theorems of Section \ref{subsec:proj}) and $\widehat{P}^k$ is the contragredient of a generalized Verma module (by Proposition \ref{prop:contra}). This means we can apply Theorem \ref{thm:gen_fus_rules} and the universal property of tensor products to get a homomorphism
\begin{equation*}
 \Pi: \widehat{V}^k_{n,e}\boxtimes\widehat{V}^k_{n',e'}\rightarrow\widehat{P}^k.
\end{equation*}
The map $\Pi$ is surjective because $\widehat{P}^k$ is generated by $\widehat{P}^k(0)$, so because $\widehat{P}^k$ is projective, it occurs as a direct summand of $\widehat{V}^k_{n,e}\boxtimes\widehat{V}^k_{n',e'}$. If $W$ is a submodule complement of $\widehat{P}^k$ in $\widehat{V}^k_{n,e}\boxtimes\widehat{V}^k_{n',e'}$, then the tensor product module has an $\NN$-grading such that
\begin{equation*}
 (\widehat{V}^k_{n,e}\boxtimes\widehat{V}^k_{n',e'})(0) = W(0)\oplus\widehat{P}^k(0).
\end{equation*}
However, Proposition \ref{prop:piY_surjective} says that $(\widehat{V}^k_{n,e}\boxtimes\widehat{V}^k_{n',e'})(0)$ is a homomorphic image of $\widehat{P}^k(0)$ (for any allowable choice of $\NN$-grading on the tensor product). Therefore $W(0)=0$ for any possible $\NN$-grading of $W$, and we conclude $W=0$. This proves the $(e+e')/k\notin\ZZ\setminus\lbrace 0\rbrace$ cases of the theorem.

Now when $(e+e')/k=\ell\in\ZZ\setminus\lbrace 0\rbrace$ and $e'/k\notin\ZZ$, we use Theorem \ref{thm:atyp_typ_fusion} and the $e+e'=0$ case of the present theorem to calculate
\begin{align*}
 \widehat{V}^k_{n,e}\boxtimes\widehat{V}^k_{n',e'} &\cong (\widehat{A}^k_{\varepsilon((e+e')/k),e+e'}\boxtimes\widehat{V}^k_{n,-e'})\boxtimes\widehat{V}^k_{n',e'}\nonumber\\
 & \cong \widehat{A}^k_{\varepsilon(\ell),\ell k}\boxtimes(\widehat{V}^k_{n,-e'}\boxtimes\widehat{V}^k_{n',e'}) \cong \widehat{A}^k_{\varepsilon(\ell),\ell k}\boxtimes\widehat{P}^k_{n+n'}.
\end{align*}
Thus it is sufficient to show that
\begin{equation}\label{eqn:atyp_proj_fusion}
 \widehat{A}^k_{\varepsilon(\ell),\ell k}\boxtimes\widehat{P}^k_n\cong\widehat{P}^k_{n+\varepsilon(\ell),\ell k}
\end{equation}
for any $\ell\in\ZZ\setminus\lbrace 0\rbrace$, $n\in\CC$. Since $\widehat{A}^k_{\varepsilon(\ell),\ell k}$ is a necessarily rigid simple current, since $\widehat{P}^k_n$ is projective in $\cO_k^{fin}$, and since tensor products of rigid with projective objects are projective (see for example \cite[Cor. 2, Appendix]{KL5}), $\widehat{A}^k_{\varepsilon(\ell),\ell k}\boxtimes\widehat{P}^k_n$ is projective in $\cO_k^{fin}$. We also have a surjection
\begin{equation*}
 \widehat{A}^k_{\varepsilon(\ell),\ell k}\boxtimes\widehat{P}^k_n \xrightarrow{\id\boxtimes q} \widehat{A}^k_{\varepsilon(\ell),\ell k}\boxtimes\widehat{A}^k_{n,0}\xrightarrow{\cong} \widehat{A}^k_{n+\varepsilon(\ell),\ell k},
\end{equation*}
where $q$ is a surjection from $\widehat{P}^k_{n}$ onto $\widehat{A}^k_{n,0}$. Since $\widehat{P}^k_{n+\varepsilon(\ell),\ell k}$ is a projective cover of $\widehat{A}^k_{n+\varepsilon(\ell),\ell k}$, it follows that $\widehat{P}^k_{n+\varepsilon(\ell),\ell k}$ is a direct summand of $\widehat{A}^k_{\varepsilon(\ell),\ell k}\boxtimes\widehat{P}^k_n$. Then because $\widehat{A}^k_{\varepsilon(\ell),\ell k}$ is a simple current, $\widehat{A}^k_{\varepsilon(\ell),\ell k}\boxtimes\widehat{P}^k_n$ has length $4$ (see for example \cite[Prop. 2.5]{CKLR}, as does $\widehat{P}^k_{n+\varepsilon(\ell),\ell k}$, and \eqref{eqn:atyp_proj_fusion} follows.
\end{proof}

We can also compute tensor products involving projective modules using the fusion rules for simple modules together with associativity of tensor products; we record them here:
\begin{cor}\label{Okgeneral}
Tensor products involving projective modules in $\cO_k^{fin}$ are as follows:
\begin{enumerate}
\item For $n,n'\in\CC$ and $\ell,\ell'\in\ZZ$,
\begin{equation*}
 \widehat{A}_{n, \ell k}^k \btimes \widehat{P}^k_{n', \ell' k} \cong \widehat{P}^k_{n+n'-\varepsilon(\ell, \ell'), (\ell+\ell')k}.
\end{equation*}

\item For $n,n'\in\CC$, $e/k\notin\ZZ$, and $\ell\in\ZZ$,
$$\widehat{V}^k_{n, e}\btimes \widehat{P}_{n', \ell' k}^k = \widehat{V}^k_{n+n'+1-\varepsilon(\ell'), e+\ell' k} \oplus 2\cdot\widehat{V}^k_{n+n'-\varepsilon(\ell'), e+\ell' k} \oplus \widehat{V}^k_{n+n'-1-\varepsilon(\ell'), e+\ell'k}.$$

\item For $n,n'\in\CC$ and $\ell,\ell'\in\ZZ$,
$$\widehat{P}^k_{n, \ell k}\btimes \widehat{P}_{n', \ell' k}^k = \widehat{P}^k_{n+n'+1-\varepsilon(\ell, \ell'), (\ell+\ell')k} \oplus 2\cdot\widehat{P}^k_{n+n'-\varepsilon(\ell, \ell'), (\ell+\ell')k} \oplus \widehat{P}^k_{n+n'-1-\varepsilon(\ell,\ell'), (\ell+\ell')k}.$$
\end{enumerate}
\end{cor}

\section{Rigidity}

In this section, we prove that the tensor supercategories $KL_k$ and $\cO_k^{fin}$ are rigid. The strategy is to first prove that all simple $V_k(\gl(1|1))$-modules are rigid, and then extend rigidity to finite-length objects using \cite[Thm. 4.4.1]{CMY2}.

\subsection{Knizhnik-Zamolodchikov equations}

In order to prove that the typical irreducible $V_k(\gl(1|1))$-modules $\widehat{V}^k_{n,e}$ are rigid, we will need explicit formulas for four-point correlation functions involving highest-weight vectors in $\widehat{V}^k_{n,e}$ and its contragredient. We will obtain these correlation functions as solutions to Knizhnik-Zamolodchikov (KZ) equations.

Recall from Proposition \ref{prop:contra} that $(\widehat{V}^k_{n,e})'\cong\widehat{V}^k_{-n,-e}$, but by an odd isomorphism. Thus it is better to take $(\widehat{V}^k_{n,e})'$ to be the parity-reversed module $\Pi(\widehat{V}^k_{-n,-e})$; that is, we take a highest-weight vector $v_{-n,-e}\in(\widehat{V}^k_{n,e})'$ to be odd. This way, intertwining operators of interest involving $\widehat{V}^k_{n,e}$ and its contragredient, in particular the ones induced by the homomorphisms
\begin{equation*}
 \widehat{V}^k_{n,e}\boxtimes\Pi(\widehat{V}^k_{-n,-e})\xrightarrow{\cong}\widehat{P}^k_0
\end{equation*}
and
\begin{equation*}
 \widehat{V}^k_{n,e}\boxtimes\Pi(\widehat{V}^k_{-n,-e})\xrightarrow{\cong}\widehat{P}^k_0\twoheadrightarrow\widehat{A}^k_{0,0},
\end{equation*}
will be even. For simplicity of notation, in what follows we will use $V$ to denote $\widehat{V}^k_{n,e}$ and $V'$ to denote $\Pi(\widehat{V}^k_{-n,-e})$.

Now let $W$ be some $V_k(\gl(1|1))$-module, $\cY_1$ an even intertwining operator of type $\binom{V}{V\,W}$, and $\cY_2$ an even intertwining operator of type $\binom{W}{V'\,V}$.
 Then we define the multivalued analytic functions
\begin{align*}
 &\Phi(z_1,z_2)(v_0,v_1,v_2,v_3): = \jiao{v_0, \cY_1(v_1,z_1)\cY_2(v_2, z_2)v_3}, \qquad\vert z_1\vert>\vert z_2\vert>0
\end{align*}
on the indicated region, where $v_0,v_2\in V'$ and $v_1,v_3\in V$ are vectors of (minimal) conformal weight $\Delta_{n,e}$. We also assume each $v_i$ is an $N_0$-eigenvector with eigenvalue $n_i$; so $n_1, n_3\in\lbrace n\pm\frac{1}{2}\rbrace$ and  $n_0,n_2\in\lbrace -n\pm\frac{1}{2}\rbrace$. It is then straightforward to use the expression \eqref{ell-1} for $L_{-1}$, the $L_{-1}$-derivative for intertwining operators, the commutator formula \eqref{eqn:intw_op_comm}, and the iterate formula \eqref{eqn:intwo_op_it} to derive the following partial differential equations (for examples of detailed calculations of this type see for example \cite{HL, Mc1, CMY2}):
\begin{prop}[Knizhnik-Zamolodchikov equations]\label{KZ}
The functions $\Phi$ satisfy the following partial differential equations:
\begin{align}
& \partial_{z_1} \Phi(z_1, z_2)(v_0, v_1, v_2, v_3)\nonumber\\
& = \left[-\frac{e}{k}\left(n_1-n_2+\frac{e}{k}\right)(z_1-z_2)^{-1}+ \frac{e}{k}\left(n_1+n_3+\frac{e}{k}\right)z_1^{-1}\right]\Phi(z_1, z_2)(v_0, v_1, v_2, v_3)\nonumber\\
% &\qquad \Phi(z_1, z_2)(v_0, v_1, v_2, v_3)\nonumber\\
 & \quad + \frac{(-1)^{|v_1|}}{k}(z_1-z_2)^{-1}\big[\Phi(z_1, z_2)(v_0, \psi^-v_1, \psi^+v_2, v_3)-\Phi(z_1, z_2)(v_0, \psi^+v_1, \psi^-v_2, v_3)\big] \nonumber\\
\label{productz1} & \quad+ \frac{(-1)^{|v_1|+|v_2|}}{k}z_1^{-1}\big[\Phi(z_1, z_2)(v_0, \psi^-v_1, v_2, \psi^+v_3)-\Phi(z_1, z_2)(v_0, \psi^+v_1, v_2, \psi^-v_3)\big]
\end{align}
% & \qquad - \frac{(-1)^{|v_1|}}{k}(z_1-z_2)^{-1}\nonumber \\
% \label{productz1}& \qquad - \frac{(-1)^{|v_1|+|v_2|}}{k}z_1^{-1}, \\
\begin{align}
& \partial_{z_2} \Phi(z_1, z_2)(v_0, v_1, v_2, v_3)\nonumber\\
&= \left[\frac{e}{k}\left(n_1-n_2+\frac{e}{k}\right)(z_1-z_2)^{-1}+ \frac{e}{k}\left(n_2-n_3-\frac{e}{k}\right)z_2^{-1}\right]\Phi(z_1, z_2)(v_0, v_1, v_2, v_3) \nonumber\\
& \quad + \frac{(-1)^{|v_1|}}{k}(z_1-z_2)^{-1}\big[\Phi(z_1, z_2)(v_0, \psi^+v_1, \psi^-v_2, v_3)-\Phi(z_1, z_2)(v_0, \psi^-v_1, \psi^+v_2, v_3)\big]\nonumber \\
\label{productz2}& \quad + \frac{(-1)^{|v_2|}}{k}z_2^{-1}\big[\Phi(z_1, z_2)(v_0, v_1, \psi^-v_2, \psi^+v_3) -\Phi(z_1, z_2)(v_0, v_1, \psi^+v_2, \psi^-v_3)\big],
\end{align}
\end{prop}

\begin{remark}
For simplicity, we will sometimes drop the dependence on $v_0$, $v_1$, $v_2$, and $v_3$ from the notation for $\Phi$.
\end{remark}

To solve the KZ equations, we reduce $\Phi$ to a one-variable function using the $L_0$-conjugation formula for intertwining operators. In fact, since the lowest conformal weight of both $V$ and $V'$ is $\Delta_{n,e}$, $L_0$-conjugation implies that
\begin{equation}\label{eqn:Phi_phi_relation}
 \Phi(z_1,z_2)=z_1^{-2\Delta_{n,e}}\phi(z_2/z_1)
\end{equation}
where
\begin{equation*}
 \phi(z):=\Phi(1,z)=\langle v_0,\cY_1(v_1,1)\cY_2(v_2,z)v_3\rangle.
\end{equation*}
It is convenient to view $\phi(z)$ as a single-valued analytic function on the simply-connected domain
\begin{equation*}
 U=\lbrace z\in\CC\,\vert\,\vert z\vert<1\rbrace\setminus(-1,0];
\end{equation*}
we fix a single-valued branch by setting $\log 1=0$ and $\log z=\log\vert z\vert+i\arg z$, where $-\pi<\arg z<\pi$.

We will need some relations between the functions $\phi$ for varying $v_1$, $v_2$, and $v_3$:
\begin{prop}
For any lowest-conformal-weight vector $v_0\in V'$, the following relations hold:
\begin{align}
\label{vanish1}
&\phi(z)(v_0,v_{n,e},v_{-n,-e},v_{n,e}) = 0,\\
\label{reduction1s}
& \phi(z)(v_0,\psi^-v_{n,e},v_{-n,-e},v_{n,e})  + \phi(z)(v_0,v_{n,e},\psi^-v_{-n,-e},v_{n,e})\nonumber\\
& \qquad = \phi(z)(v_0,v_{n,e},v_{-n,-e},\psi^-v_{n,e}).
\end{align}
\end{prop}
\begin{proof}
For $v_1=v_3=v_{n,e}$ and $v_2=v_{-n,-e}$, the KZ equations \eqref{productz1} and \eqref{productz2} become
\begin{align*}
 \partial_{z_1}\Phi(z_1,z_2) & = \left[-\frac{e}{k}\left(2n+\frac{e}{k}\right)(z_1-z_2)^{-1}+\frac{e}{k}\left(2n+1+\frac{e}{k}\right)z_1^{-1}\right]\Phi(z_1,z_2)\nonumber\\
 \partial_{z_2}\Phi(z_1,z_2) & =\frac{e}{k}\left(2n+\frac{e}{k}\right)\left[(z_1-z_2)^{-1}-z_2^{-1}\right]\Phi(z_1,z_2).
\end{align*}
From \eqref{eqn:Phi_phi_relation}, we also get the relations
\begin{align*}
 \partial_{z_1}\Phi(z_1,z_2) & = -2\Delta_{n,e} z_1^{-2\Delta_{n,e}-1}\phi(z_2/z_1)-z_1^{-2\Delta_{n,e}-2}z_2\phi'(z_2/z_1),\nonumber\\
 \partial_{z_2}\Phi(z_1,z_2) &=z_1^{-2\Delta_{n,e}-1}\phi'(z_2/z_1).
\end{align*}
Thus setting $z_1=1$, $z_2=z$ and using the definition \eqref{eqn:Delta_n,e} of $\Delta_{n,e}$, we get
\begin{align*}
 -2\Delta_{n,e}\phi(z)-z\phi'(z) & =\left(-2\Delta_{n,e}(1-z)^{-1}+2\Delta_{n,e}+\frac{e}{k}\right)\phi(z),\nonumber\\
 \phi'(z) & =2\Delta_{n,e}\left((1-z)^{-1}-z^{-1}\right)\phi(z).
\end{align*}
These two equations simplify to $\frac{e}{k}\phi(z)=0$; since $e\neq 0$, this means $\phi(z)=0$, proving \eqref{vanish1}.

For \eqref{reduction1s}, we use \eqref{vanish1}, the contragredient formula \eqref{eqn:aff_contra}, and the $r=0$ case of the commutator formula \eqref{eqn:intw_op_comm} to get
\begin{align*}
0 & =\phi(z)(\psi^-v_0,v_{n,e},v_{-n,-e},v_{n,e})\nonumber\\
&= -(-1)^{|v_0|}\phi(z)(v_0,\psi^-v_{n,e},v_{-n,-e},v_{n,e})  -(-1)^{|v_0|+|v_{n,e}|}\phi(z)(v_0,v_{n,e},\psi^-v_{-n,-e},v_{n,e})\nonumber\\
& \quad\, -(-1)^{|v_0|+|v_{n,e}|+|v_{-n,-e}|}\phi(z)(v_0,v_{n,e},v_{-n,-e},\psi^-v_{n,e}).
\end{align*}
Since $\vert v_{n,e}\vert=0$ and $\vert v_{-n,-e}\vert=1$, \eqref{reduction1s} follows.
\end{proof}

Now we derive a second-order differential equation for $\phi(z)(v_0,v_{n,e},v_{-n,-e},\psi^-v_{n,e})$. We begin with two cases of the KZ equation \eqref{productz2} specialized to $z_1=1$, $z_2=z$:
\begin{align}\label{eqn:phi1_diff_eq}
 \phi'(z) & (v_0,\psi^-v_{n,e},v_{-n,-e},v_{n,e})\nonumber\\
 &= \left[\frac{e}{k}\left(2n-1+\frac{e}{k}\right)(1-z)^{-1}-\frac{e}{k}\left(2n-\frac{e}{k}\right)z^{-1}\right]\phi(z)(v_0,\psi^-v_{n,e},v_{-n,-e},v_{n,e})\nonumber\\
 &\qquad -\frac{e}{k}(1-z)^{-1}\phi(z)(v_0,v_{n,e},\psi^- v_{-n,-e},v_{n,e})\nonumber\\
 & = 2\Delta_{n,e}\big[(1-z)^{-1}-z^{-1}\big]\phi(z)(v_0,\psi^-v_{n,e},v_{-n,-e},v_{n,e})\nonumber\\
 &\qquad-\frac{e}{k}(1-z)^{-1}\phi(z)(v_0,v_{n,e},v_{-n,-e},\psi^-v_{n,e}),
\end{align}
where the second equality uses \eqref{eqn:Delta_n,e} and \eqref{reduction1s}, and
\begin{align}\label{eqn:phi3_diff_eq}
\phi'(z) & (v_0,v_{n,e},v_{-n,-e},\psi^-v_{n,e})\nonumber\\
& = \left[\frac{e}{k}\left(2n+\frac{e}{k}\right)(1-z)^{-1}+\frac{e}{k}\left(2n-1-\frac{e}{k}\right)z^{-1}\right]\phi(z)(v_0,v_{n,e},v_{-n,-e},\psi^-v_{n,e})\nonumber\\
&\qquad - \frac{e}{k} z^{-1}\phi(z)(v_0,v_{n,e},\psi^-v_{-n,-e}, v_{n,e})\nonumber\\
& = 2\Delta_{n,e}\big[(1-z)^{-1}-z^{-1}\big]\phi(z)(v_0,v_{n,e},v_{-n,-e},\psi^-v_{n,e})\nonumber\\
&\qquad +\frac{e}{k}z^{-1}\phi(z)(v_0,\psi^-v_{n,e},v_{-n,-e},v_{n,e}).
\end{align}
We can solve \eqref{eqn:phi3_diff_eq} for $\phi(z)(v_0,\psi^-v_{n,e},v_{-n,-e},v_{n,e})$ in terms of $\phi(z)(v_0,v_{n,e},v_{-n,-e},\psi^-v_{n,e})$ and its derivative and then plug into \eqref{eqn:phi1_diff_eq}. The result is the following differential equation for $\phi(z)(v_0,v_{n,e},v_{-n,-e},\psi^-v_{n,e})$:
\begin{theorem}\label{thm:diff_eq}
 For any lowest-conformal-weight vector $v_0\in V'$, the analytic function $\phi(z)(v_0,v_{n,e},v_{-n,-e},\psi^-v_{n,e})$ is a solution to the differential equation
 \begin{align}\label{eqn:main_diff_eq}
  z(1-z)\phi''(z) &+\big[(4\Delta_{n,e}+1)-(8\Delta_{n,e}+1)z\big]\phi'(z)+4\Delta_{n,e}^2 z^{-1}\phi(z)\nonumber\\
 & +2\Delta_{n,e}(2\Delta_{n,e}-1)(1-z)^{-1}\phi(z)+\left[\left(\frac{e}{k}\right)^2-16\Delta_{n,e}^2\right]\phi(z)=0
 \end{align}
in the region $U$.
\end{theorem}

Note that $\phi(z)(v_0,v_{n,e},v_{-n,-e},\psi^-v_{n,e})$, as a series in $z$, is the series expansion of a solution to \eqref{eqn:main_diff_eq} about the regular singular point $0$. Since \eqref{eqn:main_diff_eq} is a second-order differential equation with regular singular points at $0$, $1$, and $\infty$, it can be transformed into a hypergeometric differential equation. Indeed, if
$$f(z) = z^{2\Delta_{n,e}}(1-z)^{2\Delta_{n,e}}\phi(z)$$
where $\phi(z)$ is a solution to \eqref{eqn:main_diff_eq}, then $f(z)$ satisfies the hypergeometric equation
\begin{equation}\label{secondorder2}
z(1-z)f''(z) + (1-z)f'(z) + \left(\frac{e}{k}\right)^2f(z) = 0.
\end{equation}
The solutions of \eqref{secondorder2} are well known (see for example \cite[Chap.~15]{AS} or \cite[Sec. 15.10]{DLMF}), and it follows that two linearly independent solutions of the original equation \eqref{eqn:main_diff_eq} on $U$ are
\begin{align*}
\phi^{(1)}(z) & = z^{-2\Delta_{n,e}}(1-z)^{-2\Delta_{n,e}}  {}_2F_1\left(\frac{e}{k}, -\frac{e}{k}; 1; z\right)\\
\phi^{(2)}(z) &= z^{-2\Delta_{n,e}}(1-z)^{-2\Delta_{n,e}}\left({}_2F_1\left(\frac{e}{k}, -\frac{e}{k}; 1; z\right)\log z + G(z)\right),
\end{align*}
where $G(z)$ is a power series that converges in the region $U$ (by adding a multiple of $\phi^{(1)}(z)$ if necessary, we may assume $G(z)$ has no constant term, but this is not important for us).

We also need some analytic properties of iterates of intertwining operators involving $V$ and $V'$, but we will not need explicit formulas. Let $M$ be some $V_k(\gl(1|1))$-module, $\cY^1$ an even intertwining operator of type $\binom{M}{V\,V'}$, and $\cY^2$ an even intertwining operator of type $\binom{V}{M\,V}$. The series
\begin{align*}
 &\Psi(z_0,z_2)(v_0,v_1,v_2,v_3): = \jiao{v_0, \cY^1(\cY^2(v_1, z_0)v_2, z_2)v_3},\qquad\vert z_2\vert>\vert z_0\vert>0
\end{align*}
in $z_0$ and $z_2$ converges to a multivalued analytic function in the indicated region. We then define the single-variable function
\begin{align*}
 \psi(z):=\Psi(1-z,z)=\langle v_0,\cY^1(\cY^2(v_1,1-z)v_2,z)v_3\rangle.
\end{align*}
Using the same branch of logarithm for $\log z$ and $\log(1-z)$ as we used for $\phi(z)$, $\psi(z)$ defines a single-valued analytic function on the simply-connected region
\begin{equation*}
 \widetilde{U}=\lbrace z\in\CC\,\vert\,\vert z\vert>\vert 1-z\vert>0\rbrace\setminus[1,\infty)=\left\lbrace z\in\CC\,\big\vert\,\mathrm{Re}\,z>\frac{1}{2}\right\rbrace\setminus[1,\infty).
\end{equation*}
Also, by the $L_0$-conjugation property,
\begin{align}\label{eqn:iterate_series_exp}
 \psi(z) & =z^{-2\Delta_{n,e}}\left\langle v_0,\cY^1\left(\cY^2\left(v_1,\frac{1-z}{z}\right)v_2,1\right)v_3\right\rangle\nonumber\\
 &=\left(1+\frac{1-z}{z}\right)^{2\Delta_{n,e}}\left\langle v_0,\cY^1\left(\cY^2\left(v_1,\frac{1-z}{z}\right)v_2,1\right)v_3\right\rangle
\end{align}
for lowest-conformal-weight vectors $v_0,v_2\in V'$ and $v_1,v_3\in V$. Thus the iterate of intertwining operators gives the expansion of the analytic function $\psi(z)$ as a series in $\frac{1-z}{z}$ on the region $\widetilde{U}$.

\begin{remark}
 Using the associativity of intertwining operators from \cite{HLZ6}, $\psi(z)$ is the analytic continuation to the region $\widetilde{U}$ of a corresponding product of intertwining operators $\phi(z)$. Thus the functions $\psi(z)$ satisfy the same differential equations as the functions $\phi(z)$.
\end{remark}

\subsection{Rigidity for $KL_k$ and $\cO_k^{fin}$}\label{subsec:rigidity}
In this section, we will prove that $KL_k$ and $\cO_k^{fin}$ are rigid, beginning with the simple objects. To simplify the proof, we first discuss some results on rigidity in general tensor (super)categories.

Recall that a (left) dual of an object $W$ in a tensor (super)category with unit object $\one$ is an object $W'$ equipped with an even evaluation morphism $e_W: W'\boxtimes W\rightarrow\one$ and even coevaluation $i_W: \one\rightarrow W\boxtimes W'$ such that the compositions
\begin{equation*}
 W\xrightarrow{l_W^{-1}}\one\boxtimes W\xrightarrow{i_W\boxtimes\id_W} (W\boxtimes W')\boxtimes W\xrightarrow{\cA_{W,W',W}^{-1}} W\boxtimes(W'\boxtimes W)\xrightarrow{\id_W\boxtimes e_W} W\boxtimes\one\xrightarrow{r_W} W
\end{equation*}
and
\begin{equation*}
 W'\xrightarrow{r_{W'}^{-1}}W'\boxtimes \one\xrightarrow{\id_{W'}\boxtimes i_W} W'\boxtimes (W\boxtimes W')\xrightarrow{\cA_{W',W,W'}} (W'\boxtimes W)\boxtimes W'\xrightarrow{e_W\boxtimes \id_{W'}} \one\boxtimes W'\xrightarrow{l_{W'}} W'
\end{equation*}
are identities. The object $W$ is \textit{rigid} if it also has a right dual, defined analogously; for tensor supercategories of modules for a vertex operator superalgebra, left duals are also right duals due to braiding and ribbon structure. In the following we will denote the above two rigidity compositions by $\mathfrak{R}_W$ and $\mathfrak{R}_{W'}$, respectively.

In tensor supercategories of modules for a self-contragredient vertex operator superalgebra, the contragredient $W'$ of a module $W$ satisfies the following universal property due to symmetries of intertwining operators \cite{HLZ2,Xu}: there is an even evaluation $e_W: W'\boxtimes W\rightarrow\one$ (where the unit object $\one$ is the superalgebra itself) such that for any morphism $f: X\boxtimes W\rightarrow\one$, there is a unique $\varphi: X\rightarrow W'$ such that the diagram
\begin{equation*}
 \xymatrix{
 X\boxtimes W \ar[d]_{\varphi\boxtimes\id_W} \ar[rd]^f & \\
 W'\boxtimes W \ar[r]_(.65){e_W} & \one \\
 }
\end{equation*}
commutes. In a general tensor supercategory, we say that a pair $(W', e_W)$ is a \textit{contragredient} of $W$ if it satisfies this universal property. If two objects $W$ and $X$ have contragredients, then a morphism $f: W\rightarrow X$ has a contragredient $f':X'\rightarrow W'$ defined to be the unique morphism such that the diagram
\begin{equation*}
 \xymatrix{
 X'\boxtimes W \ar[r]^{\id_{X'}\boxtimes f} \ar[d]_{f'\boxtimes\id_W} & X'\boxtimes X \ar[d]^{e_X} \\
 W'\boxtimes W \ar[r]_(.6){e_W} & \one\\
 }
\end{equation*}
commutes.
\begin{lemma}
 Let $W$ be an object of a tensor supercategory with contragredient $(W',e_W)$, and let $i_W: \one\rightarrow W\boxtimes W'$ be a morphism. Then the rigidity compositions with respect to $e_W$ and $i_W$ satisfy $\mathfrak{R}_{W'}=\mathfrak{R}_W'$.
\end{lemma}
\begin{proof}
 We need to show that $e_W\circ(\mathfrak{R}_{W'}\boxtimes\id_W)=e_W\circ(\id_{W'}\boxtimes\mathfrak{R}_W)$. The left side is the composition
 \begin{align*}
  W'  \boxtimes  W  &\xrightarrow{r_{W'}^{-1}\boxtimes\id_W} (W'\boxtimes\one)\boxtimes W\xrightarrow{(\id_{W'}\boxtimes i_W)\boxtimes\id_W} (W'\boxtimes(W\boxtimes W'))\boxtimes W\nonumber\\
  &\xrightarrow{\cA_{W',W,W'}\boxtimes\id_W} ((W'\boxtimes W)\boxtimes W')\boxtimes W\xrightarrow{(e_W\boxtimes\id_{W'})\boxtimes\id_W} (\one\boxtimes W')\boxtimes W\nonumber\\
  &\xrightarrow{l_{W'}\boxtimes\id_W} W'\boxtimes W\xrightarrow{e_W} \one.
 \end{align*}
By properties of the unit isomorphisms and naturality of associativity isomorphisms, this becomes
\begin{align*}
 W' & \boxtimes W \xrightarrow{\id_{W'}\boxtimes l_W^{-1}} W'\boxtimes(\one\boxtimes W)\xrightarrow{\id_{W'}\boxtimes(i_W\boxtimes\id_W)} W'\boxtimes((W\boxtimes W')\boxtimes W)\nonumber\\
 & \xrightarrow{\cA_{W',W\boxtimes W',W}} (W'\boxtimes(W\boxtimes W'))\boxtimes W\xrightarrow{\cA_{W',W,W'}^{-1}\boxtimes\id_W} ((W'\boxtimes W)\boxtimes W')\boxtimes W\nonumber\\
 &\xrightarrow{\cA_{W'\boxtimes W,W',W}^{-1}} (W'\boxtimes W)\boxtimes(W'\boxtimes W)\xrightarrow{e_W\boxtimes\id_{W'\boxtimes W}} \one\boxtimes(W'\boxtimes W)\xrightarrow{l_{W'\boxtimes W}} W'\boxtimes W\xrightarrow{e_W} \one.
\end{align*}
We then apply the pentagon axiom to the associativity isomorphisms and naturality of unit and associativity isomorphisms to the second $e_W$ to obtain
\begin{align*}
 W'  \boxtimes W & \xrightarrow{\id_{W'}\boxtimes l_W^{-1}} W'\boxtimes(\one\boxtimes W)\xrightarrow{\id_{W'}\boxtimes(i_W\boxtimes\id_W)} W'\boxtimes((W\boxtimes W')\boxtimes W)\nonumber\\
 & \xrightarrow{\id_{W'}\boxtimes\cA_{W,W',W}^{-1}} W'\boxtimes(W\boxtimes(W'\boxtimes W)\xrightarrow{\id_{W'}\boxtimes(\id_W\boxtimes e_W)} W'\boxtimes(W\boxtimes\one)\nonumber\\
 &\xrightarrow{\cA_{W',W,\one}} (W'\boxtimes W)\boxtimes\one\xrightarrow{e_W\boxtimes\id_\one} \one\boxtimes\one\xrightarrow{l_\one} \one.
\end{align*}
But since $l_\one=r_\one$, the last three arrows become
\begin{align*}
 r_\one\circ(e_W\boxtimes\id_\one)\circ\cA_{W',W,\one} = e_W\circ r_{W'\boxtimes W}\circ\cA_{W',W,\one}=e_W\circ(\id_{W'}\boxtimes r_W),
\end{align*}
so the entire composition is $e_W\circ(\id_{W'}\circ\mathfrak{R}_W)$.
\end{proof}

\begin{cor}\label{cor:RigW_enough}
 Let $W$ be a simple object of a tensor supercategory with contragredient $(W',e_W)$, and let $i_W:\one\rightarrow W\boxtimes W'$ be an even morphism. If the rigidity composition $\mathfrak{R}_W$ with respect to $e_W$ and $i_W$ is non-zero, then $W$ is left rigid.
\end{cor}
\begin{proof}
 Since $W$ is simple and $\mathfrak{R}_W$ is even and non-zero, we have $\mathfrak{R}_W=c\cdot\id_W$ for some non-zero scalar $c$. Then by the lemma, $$\mathfrak{R}_{W'}=\mathfrak{R}_W'=(c\cdot\id_W)'=c\cdot\id_{W'},$$
 so $(W',e_W,c^{-1}\cdot i_W)$ is a left dual of $W$.
\end{proof}

Now we can prove that the simple objects of $KL_k$ and $\cO_k^{fin}$ are rigid. For the atypical simple modules $\widehat{A}^k_{n,\ell k}$, $\ell\in\ZZ$, this is easy: By Proposition \ref{prop:contra}, $(\widehat{A}^k_{n,\ell k})'\cong\widehat{A}^k_{-n,-\ell k}$, and Theorem \ref{thm:atyp_atyp_fusion} shows that $\widehat{A}^k_{n,\ell k}\boxtimes\widehat{A}^k_{-n,-\ell k}\cong\widehat{A}^k_{0,0}\cong V_k(\gl(1|1))$. So we can take both evaluation and coevaluation to be isomorphisms, and then Corollary \ref{cor:RigW_enough} shows that $\widehat{A}^k_{n,\ell k}$ is rigid. (Actually, for $\ell\neq 0$, we should take $(\widehat{A}^k_{n,\ell k})'$ to be the parity-reversed module $\Pi(\widehat{A}^k_{-n,-\ell k})$ to ensure that evaluation and coevaluation are both even.)

Now for the typical irreducible modules, Proposition \ref{prop:contra} shows that we may take $(\widehat{V}^k_{n,e})'$ to be $\Pi(\widehat{V}^k_{-n,-e})$. We first fix a choice of evaluation and coevaluation. For the evaluation, let $\cE$ denote the intertwining operator of type $\binom{V_k(\gl(1|1))}{\Pi(\widehat{V}^k_{-n,-e})\,\widehat{V}^k_{n,e}}$ induced by the nondegenerate bilinear form
\begin{equation*}
 \langle\cdot,\cdot\rangle: \Pi(\widehat{V}^k_{-n,-e})\times\widehat{V}^k_{n,e}\rightarrow\CC
\end{equation*}
such that $\langle\psi^-v_{-n,-e},v_{n,e}\rangle=\langle v_{-n,-e},\psi^- v_{n,e}\rangle =1$. In particular, for lowest-conformal-weight vectors $v'\in\Pi(V_{-n,-e})$, $v\in V_{n,e}$, we have
\begin{equation*}
 \cE(v',x)v \in x^{-2\Delta_{n,e}}\big( \langle v',v\rangle\one+x\,V_k(\gl(1|1))[[x]]\big).
\end{equation*}
We then define the evaluation $\varepsilon:\Pi(\widehat{V}^k_{-n,-e})\boxtimes\widehat{V}^k_{n,e}\rightarrow V_k(\gl(1|1))$ to be the unique homomorphism such that $\varepsilon\circ\cY_\boxtimes=\cE$, where $\cY_\boxtimes$ denotes the canonical even tensor product intertwining operator of type $\binom{\Pi(\widehat{V}^k_{-n,-e})\boxtimes\widehat{V}^k_{n,e}}{\Pi(\widehat{V}^k_{-n,-e})\,\widehat{V}^k_{n,e}}$.

For the coevaluation, we first note that there is an even $\gl(1|1)$-module homomorphism
$A_0\rightarrow V_{n,e}\otimes \Pi(V_{-n,-e})$ defined by
\begin{equation*}
 \one \mapsto \psi^- v_{n,e}\otimes v_{-n,-e}+v_{n,e}\otimes\psi^- v_{-n,-e}.
\end{equation*}
We can compose this with the $\gl(1|1)$-homomorphism
\begin{equation*}
 \pi(\cY_\boxtimes): V_{n,e}\otimes\Pi(V_{-n,-e})\rightarrow(\widehat{V}^k_{n,e}\boxtimes\Pi(\widehat{V}^k_{-n,-e}))(0)
\end{equation*}
of Section \ref{subsec:gen_fus_rules}, and then use the universal property of induced $\widehat{\gl(1|1)}$-modules to extend to a homomorphism
\begin{equation*}
 i: V_k(\gl(1|1))\rightarrow\widehat{V}^k_{n,e}\boxtimes\Pi(\widehat{V}^k_{-n,-e}).
\end{equation*}
By definition,
\begin{equation*}
 i(\one) =\pi_0\left(\cY_\boxtimes(\psi^-v_{n,e},1)v_{-n,-e}+\cY_\boxtimes(v_{n,e},1)\psi^- v_{-n,-e}\right).
\end{equation*}
Equivalently, $i(\one)$ is the coefficient of $x^{-2\Delta_{n,e}}$ in $\cY_\boxtimes(\psi^-v_{n,e},x)v_{-n,-e}+\cY_\boxtimes(v_{n,e},x)\psi^- v_{-n,-e}.$

Now Corollary \ref{cor:RigW_enough} implies that $\widehat{V}^k_{n,e}$ will be rigid if the rigidity composition
\begin{equation*}
 \mathfrak{R}=r\circ(\id\boxtimes\varepsilon)\circ\cA^{-1}\circ(i\boxtimes\id)\circ l^{-1}
\end{equation*}
is non-zero. We shall prove this by showing $\langle \psi^-v_{-n,-e},\mathfrak{R}(v_{n,e})\rangle\neq 0$. From the definitions, $(i\boxtimes \id)\circ l^{-1}(v_{n,e})$ is the coefficient of $\left(\frac{1-x}{x}\right)^{-2\Delta_{n,e}}$ in the series
\begin{align*}
 \cY_\boxtimes\left(\cY_\boxtimes\left(\psi^- v_{n,e},\frac{1-x}{x}\right)v_{-n,-e},1\right)v_{n,e} +\cY_\boxtimes\left(\cY_\boxtimes\left(v_{n,e},\frac{1-x}{x}\right)\psi^-v_{-n,-e},1\right)v_{n,e}.
\end{align*}
Note also that $\left(\frac{1-x}{x}\right)^{-2\Delta_{n,e}}$ is the lowest power of $\frac{1-x}{x}$ in this series since $0$ is the lowest conformal weight of $\widehat{V}_{n,e}^k\btimes\Pi(\widehat{V}_{-n,-e}^k)$ (recall Lemma \ref{lem:usefullemma_1} and Remark \ref{rem:useful}).
We take $x$ to be a real number in the interval $(\frac{1}{2},1)=U\cap\widetilde{U}\cap\RR$, and then recalling \eqref{eqn:iterate_series_exp}, we find that  $\langle\psi^- v_{-n,-e},\mathfrak{R}(v_{n,e})\rangle$ is the coefficient of $\left(\frac{1-x}{x}\right)^{-2\Delta_{n,e}}$ in the expansion of the following analytic function as a series in $\frac{1-x}{x}$ and $\ln\left(\frac{1-x}{x}\right)$ on $(\frac{1}{2},1)$:
\begin{align}\label{eqn:rig_calc}
& \left\langle \psi^-v_{-n,-e},\overline{r\circ(\id\boxtimes\varepsilon)\circ\cA^{-1}}\circ\cY_\boxtimes(\cY_\boxtimes(\psi^-v_{n,e},1-x)v_{-n,-e},x)v_{n,e}\right\rangle\nonumber\\
 &\qquad\qquad +\left\langle \psi^-v_{-n,-e},\overline{r\circ(\id\boxtimes\varepsilon)\circ\cA^{-1}}\circ\cY_\boxtimes(\cY_\boxtimes(v_{n,e},1-x)\psi^-v_{-n,-e},x)v_{n,e}\right\rangle\nonumber\\
 &\qquad =\left\langle\psi^-v_{-n,-e},\overline{r\circ(\id\boxtimes\varepsilon)}\circ\cY_\boxtimes(\psi^-v_{n,e},1)\cY_\boxtimes(v_{-n,-e},x)v_{n,e}\right\rangle\nonumber\\
 &\qquad\qquad +\left\langle\psi^-v_{-n,-e},\overline{r\circ(\id\boxtimes\varepsilon)}\circ\cY_\boxtimes(v_{n,e},1)\cY_\boxtimes(\psi^-v_{-n,-e},x)v_{n,e}\right\rangle\nonumber\\
 &\qquad =\left\langle\psi^-v_{-n,-e},\Omega(Y_{\widehat{V}^k_{n,e}})(v_{n,e},1)\cE(v_{-n,-e},x)\psi^-v_{n,e}\right\rangle,
\end{align}
where we have used \eqref{reduction1s} for the last equality, and $\Omega(Y_{\widehat{V}_{n,e}^k})$ is the intertwining operator of type $\binom{\widehat{V}^k_{n,e}}{\widehat{V}^k_{n,e}\,V_k(\gl(1|1))}$ obtained from the vertex operator by skew-symmetry.

By Theorem \ref{thm:diff_eq}, \eqref{eqn:rig_calc} is a solution to the differential equation \eqref{eqn:main_diff_eq}. As a series in $x$, it is non-logarithmic with lowest-degree term
\begin{equation*}
 \langle\psi^-v_{-n,-e},v_{n,e}\rangle\langle v_{-n,-e},\psi^-v_{n,e}\rangle x^{-2\Delta_{n,e}} =x^{-2\Delta_{n,e}},
\end{equation*}
so \eqref{eqn:rig_calc}  is the fundamental basis solution
\begin{equation*}
 \phi^{(1)}(x)=x^{-2\Delta_{n,e}}(1-x)^{-2\Delta_{n,e}} {}_2 F_1\left(\frac{e}{k},-\frac{e}{k};1;x\right).
\end{equation*}
Thus we are reduced to showing that in the expansion of ${}_2 F_1\left(\frac{e}{k},-\frac{e}{k};1;x\right)$ as a series in $\frac{1-x}{x}$ and $\ln\left(\frac{1-x}{x}\right)$ on the interval $(\frac{1}{2},1)$, the contant term is non-zero. Indeed, from \cite[Eq. 15.8.11]{DLMF}, the constant term in this expansion is $\left[\Gamma\left(1+\frac{e}{k}\right)\Gamma\left(1-\frac{e}{k}\right)\right]^{-1}=\frac{\sin(\pi e/k)}{\pi e/k}\neq 0$ since $e/k\notin\ZZ$. We conclude that $\widehat{V}^k_{n,e}$ is rigid.

We now extend rigidity from simple objects to the full supercategories $KL_k$ and $\cO_k^{fin}$. First recall that all modules in both categories have finite length (in the sense that every module has a finite filtration of $\ZZ_2$-graded submodules such that the consecutive quotients are $\ZZ_2$-graded simple $V_k(\gl(1|1))$-modules). From this, it is clear that $KL_k$ is closed under $\ZZ_2$-graded submodules and quotients, and the same is also clear for $\cO_k^{fin}$. Moreover, since taking contragredients is an exact functor, $KL_k$ is closed under contragredients; $\cO_k^{fin}$ is also closed under contragredients by the $a_r=E_0,N_0$ cases of \eqref{eqn:aff_contra}. These are the conditions needed to apply the (straightforward superalgebra generalization of) \cite[Thm 4.4.1]{CMY2}, which states that if simple objects are rigid, then so are finite-length objects:
\begin{theorem}\label{rigidityofcategoryO}
 The tensor supercategories $KL_k$ and $\cO_k^{fin}$ are rigid; moreover, they are braided ribbon tensor supercategories with even natural twist isomorphism $\theta=e^{2\pi i L_0}$.
\end{theorem}

\section{Simple current extensions of \texorpdfstring{$V_1(\gl(1|1))$}{V1(gl(1|1))}}
There are a few families of vertex operator superalgebras extending $V_k(\gl(1|1))$. These are all infinite-order simple current extensions, and we refer to \cite{CKL} for the general theory.

The following simple current extensions of $V_1(\gl(1|1))$ are considered in \cite{CR1}:
\begin{equation}\label{generalextension}
\frW_{n+\frac{1}{2}, \ell} = \widehat{A}_{0,0} \oplus \bigoplus_{m \geq 1}(\widehat{A}_{n+\frac{1}{2}, \ell}^{\btimes m}\oplus \widehat{A}_{-n-\frac{1}{2}, -\ell}^{\btimes m})
\end{equation}
for $\ell\in \ZZ$ and $n \in \RR$ such that $|\ell| \leq 2\Delta_{n+\frac{1}{2},\ell}$ and $2n\ell \in \ZZ$. In particular, it is shown that as vertex superalgebras
\begin{equation*}
\frW_{\frac{1}{2},-2} \cong V_{-\frac{1}{2}}(\sl(2|1)), \;\;\; \frW_{\frac{1}{2},1} \cong V_{1}(\sl(2|1)),
\end{equation*}
where $V_{-\frac{1}{2}}(\sl(2|1))$ and $V_{1}(\sl(2|1))$ are the simple vertex operator superalgebras of $\sl(2|1)$ at levels $-\frac{1}{2}$ and $1$, respectively. These superalgebras are thus objects of the direct limit completion $\Ind(\cO_1^{fin})$ of the category $\cO_1^{fin}$ of $V_1(\gl(1|1))$-modules. The direct limit completion is the category of all generalized $V_1(\gl(1|1))$-modules that are unions of submodules in $\cO_1^{fin}$; the existence of vertex and braided tensor category structures on $\Ind(\cO_1^{fin})$ follows after verifying the conditions of \cite[Thm. 1.1]{CMY1}:
\begin{enumerate}
 \item The vertex operator superalgebra $V_1(\gl(1|1))$ is an object of $\cO_1^{fin}$.
 \item The category $\cO_1^{fin}$ is closed under submodules, quotients, and finite direct sums.
 \item Every module in $\cO_1^{fin}$ is finitely generated (since every module in $\cO_1^{fin}$ has finite length).
 \item The category $\cO_1^{fin}$ admits vertex and braided tensor category structures by Theorem \ref{thm:Okfin_tens_cat}.
 \item For any intertwining operator $\cY$ of type $\binom{X}{W_1\,W_2}$ where $W_1$, $W_2$ are modules in $\cO_1^{fin}$ and $X$ is a generalized $V_1(\gl(1|1))$-module in $\Ind(\cO_1^{fin})$, the submodule $\Im\,\cY\subseteq X$ is an object of $\cO_1^{fin}$. (By \cite[Cor. 2.13]{CMY1}, $\Im\,\cY$ is $C_1$-cofinite, that is, an object of $KL_1$. Then $\Im\,\cY$ is an object of $\cO_1^{fin}$ by the $a=E, N$, $r=0$ cases of \eqref{eqn:intw_op_comm}.)
\end{enumerate}
For more details on the braided tensor category structure on $\Ind(\cO_1^{fin})$, see \cite{CMY1}.

%
% See \cite{CMY1} for details on the vertex and braided tensor category structures on the direct limit completion, which . The existence of vertex and braided tensor category structures on $\Ind(\cO_1^{fin})$ follows after verifying

We can now study $V_{-\frac{1}{2}}(\sl(2|1))$-modules using the induction functor
\begin{equation*}
\cF: \cO_1^{fin} \rightarrow \Rep\,V_{-\frac{1}{2}}(\sl(2|1))
\end{equation*}
from \cite{CKM}, where $\Rep\,V_{-\frac{1}{2}}(\sl(2|1))$ is the category of (not necessarily local) $V_{-\frac{1}{2}}(\sl(2|1))$-modules in $\Ind(\cO_1^{fin})$. By \cite[Lem.~2.65]{CKM}, a $V_1(\gl(1|1))$-module induces to a local $V_{-\frac{1}{2}}(\sl(2|1))$-module, that is, an object of the braided tensor subcategory $\Rep^0\,V_{-\frac{1}{2}}(\sl(2|1))$, if and only if its monodromy with $V_{-\frac{1}{2}}(\sl(2|1))$ is trivial. Note that \eqref{generalextension} gives
\begin{equation*}
\frW_{\frac{1}{2}, -2} = \widehat{A}_{0,0} \oplus \bigoplus_{m \geq 1}(\widehat{A}_{m-\frac{1}{2}, -2m}\oplus \widehat{A}_{-m+\frac{1}{2}, 2m})=\bigoplus_{m\in\ZZ} \widehat{A}_{m-\varepsilon(m),-2m}.
\end{equation*}
So by the fusion rules in Theorem \ref{thm:atyp_atyp_fusion} and the balancing equation with twist $\theta = e^{2\pi i L_0}$, the monodromy of $\widehat{A}_{n,\ell}$ with $\widehat{A}_{m-\varepsilon(m), -2m}$ for $n \in \CC$, $\ell \in \ZZ$, $m \in \ZZ$ is
\begin{align*}
\cM_{\widehat{A}_{n,\ell}, \widehat{A}_{m-\varepsilon(m), -2m}} &= \theta_{\widehat{A}_{n,\ell}\btimes\widehat{A}_{m-\varepsilon(m), -2m}}\circ (\theta_{\widehat{A}_{n,\ell}}^{-1}\btimes \theta_{\widehat{A}_{m-\varepsilon(m), -2m}}^{-1})\nonumber\\
& = \theta_{\widehat{A}_{n+m-\varepsilon(\ell)+\varepsilon(\ell-2m), \ell-2m}}\circ (\theta_{\widehat{A}_{n,\ell}}^{-1}\btimes \theta_{\widehat{A}_{m-\varepsilon(m), -2m}}^{-1})\nonumber\\
& = e^{2\pi i(\Delta_{n+m-\varepsilon(\ell)+\varepsilon(\ell-2m), \ell-2m}-\Delta_{n,\ell}-\Delta_{m-\varepsilon(m), -2m})}.
\end{align*}
From this, we see that $\widehat{A}_{n,\ell}$ induces to a local module if and only if $2n \in \ZZ$. Similarly by the fusion rules in Theorem \ref{thm:atyp_typ_fusion}, the monodromy of the typical module $\widehat{V}_{n,e}$ for $n \in \CC, e \notin \ZZ$ with $\widehat{A}_{m-\varepsilon(m), -2m}$ is
\begin{align*}
\cM_{\widehat{V}_{n,e}, \widehat{A}_{m-\varepsilon(m), -2m}} &= \theta_{\widehat{V}_{n,e}\btimes\widehat{A}_{m-\varepsilon(m), -2m}}\circ (\theta_{\widehat{V}_{n,e}}^{-1}\btimes \theta_{\widehat{A}_{m-\varepsilon(m), -2m}}^{-1})\nonumber\\
& = \theta_{\widehat{V}_{n+m, e-2m}}\circ (\theta_{\widehat{V}_{n,e}}^{-1}\btimes \theta_{\widehat{A}_{m-\varepsilon(m), -2m}}^{-1})\nonumber\\
& = e^{2\pi i(\Delta_{n+m, e-2m}-\Delta_{n,e}-\Delta_{m-\varepsilon(m), -2m})},
\end{align*}
From this, we see that $\widehat{V}_{n,e}$ induces to a local module if and only if $2n+e \in \ZZ$. Moreover, we can determine the simple objects of $\Rep^0\,V_{-\frac{1}{2}}(\sl(2|1))$ as follows:
\begin{prop}\label{prop:simple}
The simple objects of $\Rep^0\,V_{-\frac{1}{2}}(\sl(2|1))$ are of the form $\cF(S)$, where $S$ is either $\widehat{A}_{n,\ell}$ for $n \in \frac{1}{2}\ZZ$, $\ell \in \ZZ$ or $\widehat{V}_{n,e}$ for $n \in \CC$, $e \notin \ZZ$ such that $2n+e \in \ZZ$. Moreover, $\cF(S)\cong \cF(S')$ if and only if there exists $m \in \ZZ$ such that
\begin{equation}\label{simplecurrentequiv}
S' \cong S \btimes \widehat{A}_{m-\varepsilon(m), -2m}.
\end{equation}
\end{prop}
\begin{proof}
Let $X$ be a simple object of $\Rep^0\,V_{-\frac{1}{2}}(\sl(2|1))$ and let $\cG$ be the restriction functor $\Rep^0\,V_{-\frac{1}{2}}(\sl(2|1)) \rightarrow \Ind(\cO_1^{fin})$. As an object of $\Ind(\cO_1^{fin})$, $\cG(X)$ is the union of $\cO_1^{fin}$-submodules; thus because every non-zero object of $\cO_1^{fin}$ contains an irreducible submodule, $\cG(X)$ contains an irreducible submodule $S$. By Frobenius reciprocity,
\[
\Hom(\cF(S), X) \cong \Hom(S, \cG(X)) \neq 0,
\]
so if $\cF(S)$ is simple, then $\cF(S)\cong X$.
Indeed, by examining the $E_0$-eigenvalues of $\widehat{A}_{m-\varepsilon(m), -2m}\btimes S$,
\[
\widehat{A}_{m-\varepsilon(m), -2m}\btimes S \ncong \widehat{A}_{m'-\varepsilon(m'), -2m'}\btimes S
\]
for $m \neq m'$, and then \cite[Prop.~4.4]{CKM} shows $\cF(S)$ is simple. Moreover, we have seen that $\cF(S)$ is an object in $\Rep^0\,V_{-\frac{1}{2}}(\sl(2|1))$ if and only if $S = \widehat{A}_{n,\ell}$ for $n \in \frac{1}{2}\ZZ$, $\ell \in \ZZ$ or $\widehat{V}_{n,e}$ for $n \in \CC$, $e \notin \ZZ$ such that $2n+e \in \ZZ$.

The condition \eqref{simplecurrentequiv} follows from Frobenius reciprocity.
\end{proof}

\begin{remark}
As $V_1(\gl(1|1))$-modules,
\[
\cF(\widehat{A}_{n,\ell}) = \bigoplus_{m\in \ZZ}\widehat{A}_{n+m-\varepsilon(\ell)+\varepsilon(\ell-2m), \ell-2m},
\]
and
\[
\cF(\widehat{V}_{n,e}) = \bigoplus_{m\in \ZZ}\widehat{V}_{n+m, e-2m}.
\]
Since the lowest conformal weights of the summands
$\widehat{A}_{n+m-\varepsilon(\ell)+\varepsilon(n-2m), \ell-2m}$
and $\widehat{V}_{n+m, e-2m}$ are both linear functions of $m$, most of the induced simple objects $\cF(\widehat{A}_{n,\ell})$ and $\cF(\widehat{V}_{n,e})$ are not lowest-weight modules for $V_{-\frac{1}{2}}(\sl(2|1))$.

The character of $\widehat{V}_{n,e}$ is
\[
\text{ch}[\widehat{V}_{n,e}](z, y; q) =  q^{\Delta_{n,e}}y^ez^n \prod_{i=0}^\infty\frac{(1+zq^{i+1})(1+z^{-1}q^i)}{(1-q^{i+1})^2}
\]
so that we get
\begin{equation}\label{eq:RHWch}
\begin{split}
\text{ch}[\cF(\widehat{V}_{n,e})](z, y; q) &=  \text{ch}[\widehat{V}_{n,e}](z, y; q)  \sum_{m \in\ZZ }     q^{-m(2n+e)}y^{-2m}z^m.
\end{split}
\end{equation}
Thus $\cF(\widehat{V}_{n,e})$ has an infinite-dimensional lowest-conformal-weight space if $2n+e=0$ and has unbounded conformal weights otherwise. These are examples of relaxed highest-weight modules and their images under spectral flow.
\end{remark}

Now that we have determined the simple modules in $\Rep^0\,V_{-\frac{1}{2}}(\sl(2|1))$, we determine their projective covers. First, we need a lemma:
\begin{lemma}\label{lem:ind_proj}
 If $P$ is projective in $\cO_1^{fin}$, then $\cF(P)$ is projective in $\Rep\,V_{-\frac{1}{2}}(\sl(2|1))$.
\end{lemma}
\begin{proof}
 We first verify that $P$ is projective in $\Ind(\cO_1^{fin})$. Thus suppose we have a diagram
 \begin{equation*}
  \xymatrix{
  & P \ar[d]^q \\
  X \ar[r]_p & Y \\
  }
 \end{equation*}
in $\Ind(\cO_1^{fin})$ with $p$ surjective. Since $\cO_1^{fin}$ is closed under quotients, $\Im\,q\subseteq Y$ is a (finitely-generated) object of $\cO_1^{fin}$. Then since $X$, as an object of $\Ind(\cO_1^{fin})$, is the union of its $\cO_1^{fin}$-submodules, there are finitely many $\cO_1^{fin}$-submodules $\lbrace W_i\rbrace$ which contain preimages of a generating set for $\Im\,q$. Because $\cO_1^{fin}$ is closed under finite direct sums and quotients, the submodule $W=\sum W_i$ is an object of $\cO_1^{fin}$, and $p(W)\subseteq Y$ is an $\cO_1^{fin}$-submodule that contains $\Im\,q$. Now because $P$ is projective in $\cO_1^{fin}$, there is a map $f: P\rightarrow W$ such that $q=p\vert_W\circ f$. Interpreting $f$ as a map into $X$, we conclude that $P$ is projective in $\Ind(\cO_1^{fin})$.

Now, as in \cite[Thm. 17]{ACKR}, Frobenius reciprocity implies that $\cF(P)$ is projective in $\Rep\,V_{-\frac{1}{2}}(\sl(2|1))$. Specifically, $\Hom_{\Rep\,V_{-\frac{1}{2}}(\sl(2|1))}(\cF(P),\cdot)$ is exact since it is the composition of two exact functors: the restriction functor $\cG: \Rep\,V_{-\frac{1}{2}}(\sl(2|1))\rightarrow\Ind(\cO_1^{fin})$ and the $\Hom$ functor $\Hom_{\Ind(\cO_1^{fin})}(P,\cdot)$.
\end{proof}

Now we can prove:
\begin{prop}\label{prop:proj}
 Suppose $S$ is a simple $V_1(\gl(1|1))$-module with projective cover $P_S$ in $\cO_1^{fin}$ such that $\cF(S)$ is an object of $\Rep^0\,V_{-\frac{1}{2}}(\sl(2|1))$. Then $\cF(P_S)$ is a projective cover of $\cF(S)$ in $\Rep^0\,V_{-\frac{1}{2}}(\sl(2|1))$.
\end{prop}
\begin{proof}
 By \cite[Thm. 1.4(1)]{CKL}, the conformal weight criterion for $\cF(P_S)$ to be local is the same as that for $\cF(S)$, so $\cF(P_S)$ is an object of $\Rep^0\,V_{-\frac{1}{2}}(\sl(2|1))$ if $\cF(S)$ is, and then $\cF(P_S)$ is projective in $\Rep^0\,V_{-\frac{1}{2}}(\sl(2|1))$ by the preceding lemma. Furthermore, because induction is exact, the canonical surjection $p: P_S\rightarrow S$ induces to a surjection $\cF(p):\cF(P_S)\rightarrow\cF(S)$.

 Now if $S$ is a typical simple $V_1(\gl(1|1))$-module, then $P_S=S$ and it is clear that $\cF(P_S)$ is a projective cover of $\cF(S)$. If $S=\widehat{A}_{n,\ell}$ is atypical  with projective cover $\widehat{P}_{n,\ell}$ in $\cO_1^{fin}$, we need to show that for any surjection $q: P\rightarrow \cF(\widehat{A}_{n,\ell})$ in $\Rep^0\,V_{-\frac{1}{2}}(\sl(2|1))$ with $P$ projective, there is a surjection $f: P\rightarrow\cF(\widehat{P}_{n,\ell})$ such that $q=\cF(p)\circ f$. Indeed, because $P$ and $\cF(\widehat{P}_{n,\ell})$ are both projective, we have maps $f: P\rightarrow\cF(\widehat{P}_{n,\ell})$ and $g: \cF(\widehat{P}_{n,\ell})\rightarrow P$ such that the diagrams
 \begin{equation*}
 \xymatrix{
  & P \ar[ld]_{f} \ar[d]^q \\
  \cF(\widehat{P}_{n,\ell}) \ar[r]_{\cF(p)} & \cF(\widehat{A}_{n,\ell}) \\
  } \qquad\qquad
  \xymatrix{
  & P  \ar[d]^q \\
  \cF(\widehat{P}_{n,\ell}) \ar[ru]^{g} \ar[r]_{\cF(p)} & \cF(\widehat{A}_{n,\ell}) \\
  }
 \end{equation*}
commute. Since induction is exact, $\cF(\widehat{P}_{n,\ell})$ has finite length; so if we can show that $\cF(\widehat{P}_{n,\ell})$ is also indecomposable, then Fitting's Lemma implies that $f\circ g$ is either nilpotent or an isomorphism. It cannot be nilpotent because
\begin{equation*}
 \cF(p)\circ(f\circ g)^N =\cF(p)\neq 0
\end{equation*}
for all $N\in\NN$, so it is an isomorphism. It follows that $f$ is surjective (and $g$ is injective).

It remains to show that $\cF(\widehat{P}_{n,\ell})$ is indecomposable. It is enough to show that its socle is $\cF(\widehat{A}_{n,\ell})$ (and in particular is simple), because then the intersection of any two non-zero submodules in $\cF(\widehat{P}_{n,\ell})$ will contain the socle. By Frobenius reciprocity and Corollary \ref{Okgeneral}, for any simple $V_1(\gl(1|1))$-module $S$ such that $\cF(S)$ is local,
\begin{align*}
 \dim\Hom_{\Rep^0\,V_{-\frac{1}{2}}(\sl(2|1))}(\cF(S),\cF(\widehat{P}_{n,\ell})) & =\sum_{m\in\ZZ}\dim\Hom_{\cO_1^{fin}}(S,\widehat{A}_{m-\varepsilon(m),-2m}\boxtimes\widehat{P}_{n,\ell})\nonumber\\
 & \hspace{-2em}=\left\lbrace\begin{array}{ll}
                 1 & \text{if}\,\, S\cong\widehat{A}_{m-\varepsilon(m),-2m}\boxtimes\widehat{A}_{n,\ell}\,\,\text{for some}\,\,m\in\ZZ \\
                 0 & \text{otherwise}
                \end{array}
                \right. .
\end{align*}
Thus $\cF(S)$ occurs as a submodule of $\cF(\widehat{P}_{n,\ell})$, with multiplicity one, if and only if $\cF(S)\cong\cF(\widehat{A}_{n,\ell})$ (recall \eqref{simplecurrentequiv}). Consequently, $\mathrm{Soc}\,\cF(\widehat{P}_{n,\ell})\cong\cF(\widehat{A}_{n,\ell})$ as required.
\end{proof}

For the vertex operator superalgebra $V_{1}(\sl(2|1))$,
first \eqref{generalextension} gives
\begin{equation}
\frW_{\frac{1}{2}, 1} = \widehat{A}_{0,0} \oplus \bigoplus_{m \in \ZZ_{\geq 1}}(\widehat{A}_{\frac{1}{2}, m}\oplus \widehat{A}_{-\frac{1}{2}, -m})=\bigoplus_{m\in\ZZ} \widehat{A}_{\varepsilon(m),m}.
\end{equation}
Then similar analysis as above gives the simple objects in $\Rep^0\,V_{1}(\sl(2|1))$ and their projective covers:
\begin{prop}
The simple objects of $\Rep^0\,V_{1}(\sl(2|1))$ are of the form $\cF(S)$, where $S$ is either $\widehat{A}_{n,\ell}$ for $n \in \frac{1}{2} + \ZZ$, $\ell \in \ZZ\setminus \{0\}$ or $n \in \ZZ$, $\ell = 0$, or $\widehat{V}_{n,e}$ for $n \in \CC$, $e \notin \ZZ$ such that $n+e \in \frac{1}{2}+\ZZ$. For any simple $V_1(\gl(1|1))$-module $S$ such that $\cF(S)$ is an object of $\Rep^0\,V_{1}(\sl(2|1))$, $\cF(P_S)$ is a projective cover of $\cF(S)$ in $\Rep^0\,V_{1}(\sl(2|1))$, where $P_S$ is a projective cover of $S$ in $\cO_1^{fin}$.
\end{prop}

\begin{remark}
As $V_1(\gl(1|1))$-modules,
\[
\cF(\widehat{A}_{n,\ell}) = \bigoplus_{m\in \ZZ}\widehat{A}_{n-\varepsilon(\ell)+\varepsilon(\ell+m), \ell+m},
\]
and
\[
\cF(\widehat{V}_{n,e}) = \bigoplus_{m\in \ZZ}\widehat{V}_{n, e+m}.
\]
Note that the lowest conformal weights of the summands
$\widehat{A}_{n-\varepsilon(\ell)+\varepsilon(\ell+m), \ell+m}$
and $\widehat{V}_{n, e+m}$ are both quadratic functions of $m$ with leading term $\frac{1}{2}m^2$. Thus the induced simple objects $\cF(\widehat{A}_{n,\ell})$ and $\cF(\widehat{V}_{n,e})$ are all lowest-weight modules for $V_{1}(\sl(2|1))$.
\end{remark}

\bigskip

\noindent
Department of Mathematical and Statistical Sciences, University of Alberta, Edmonton, Alberta T6G 2R3, Canada\\
\emph{E-mail:} \texttt{creutzig@ualberta.ca, jinwei2@ualberta.ca}

\medskip

\noindent
Yau Mathematical Sciences Center, Tsinghua University, Beijing 100084, China\\
\emph{E-mail:} \texttt{rhmcrae@tsinghua.edu.cn}

\end{document}